\documentclass[review]{elsarticle}
\usepackage[usenames]{color}
\usepackage{extarrows}
\usepackage{lineno,hyperref}
\modulolinenumbers[5]

\journal{..}

\usepackage{amsfonts}
\usepackage{amsmath, cases}
\usepackage{indentfirst}
\usepackage{graphicx}
\usepackage{latexsym,bm, amsthm}

\usepackage{epsfig}
\usepackage{latexsym}
\usepackage{amssymb}
\allowdisplaybreaks[2]
\begin{document}

\begin{frontmatter}

\author[1]{Hongxing Wang\corref{mycorrespondingauthor}}
\cortext[mycorrespondingauthor]{Corresponding author}
\ead{winghongxing0902@163.com}

\title{The  {${\mathfrak{m}}$}-core-EP inverse
in Minkowski space}

\author[1,2]{Hui Wu}
\ead{huiwumath168@163.com}

\author[1]{Xiaoji Liu}
\ead{xiaojiliu72@126.com}

\address[1]{School of Mathematics and Physics,
Guangxi University for Nationalities,
Nanning 530006, China}

\address[2]{School of Mathematics and Physics,
Hechi University,
Yizhou 546300, China}

\begin{abstract}
In this paper, we introduce the {${\mathfrak{m}}$}-core-EP inverse in  Minkowski space, consider its properties, and get several sufficient and necessary conditions for the existence of {{the {${\mathfrak{m}}$}-core-EP inverse}}. We give the ${\mathfrak{m}}$-core-EP decomposition  in  Minkowski space, and note that not every square matrix has the decomposition. Furthermore, by applying the {${\mathfrak{m}}$}-core-EP inverse and the ${\mathfrak{m}}$-core-EP decomposition,  we  introduce the {${\mathfrak{m}}$}-core-EP  order and give some characterizations of it.
\end{abstract}


\end{frontmatter}

\linenumbers

\section{Introduction}
\numberwithin{equation}{section}
\newtheorem{theorem}{T{\scriptsize HEOREM}}[section]
\newtheorem{lemma}[theorem]{L{\scriptsize  EMMA}}
\newtheorem{corollary}[theorem]{C{\scriptsize OROLLARY}}
\newtheorem{proposition}[theorem]{P{\scriptsize ROPOSITION}}
\newtheorem{remark}{R{\scriptsize  EMARK}}[section]
\newtheorem{definition}{D{\scriptsize  EFINITION}}[section]
\newtheorem{algorithm}{A{\scriptsize  LGORITHM}}[section]
\newtheorem{example}{E{\scriptsize  XAMPLE}}[section]
\newtheorem{problem}{P{\scriptsize  ROBLEM}}[section]
\newtheorem{assumption}{A{\scriptsize  SSUMPTION}}[section]

In this paper, we use the following notations. The symbol $\mathbb{C}_{n, n}$ is the set of $n\times n$ matrices with complex entries.
 $A^\ast$, ${\mathcal{R}}\left( {A} \right)$ and
${\rm rk}\left( A \right)$  represent the conjugate transpose, range {{space (or }}column space) and
rank of $A\in\mathbb{C}_{n, n}$, respectively. The smallest positive integer $k$,
which satisfies
  ${\rm rk}\left(A^{k+1}\right)
  ={\rm rk}\left(A^k\right)$,
   is called the index of $A$
   and
   is denoted by ${\rm Ind}(A)$.
   In particular,
\begin{align*}
\mathbb{C}^{\tiny{\mbox{\rm CM}}}_n
=
\left\{A\mid A\in \mathbb{C}_{n,n},\
{\rm rk}(A^2)={\rm rk}(A)\right\}.
\end{align*}

In 1907, Hermann Minkowski proposed the Minkowski space (shorted as $\mathcal{M}$).
In 2000,
 Meenakshi \cite{Meenakshi2000pnsaa1}
  introduced generalized inverse into  $\mathcal{M}$,
and got  its existence conditions.
The Minkowski metric matrix can be written as
\begin{align}
 \label{WWL-MCE-1}
G
=
 \left[{\begin{matrix}
 1
    & 0   \\
 0   & -I_{n-1}   \\
\end{matrix} }\right],
\
 G=G^{*}
\
{\rm and }
\
G^{2}=I_{n}.
\end{align}
The Minkowski adjoint of
a matrix $A\in \mathbb{C}_{n,n}$
is defined as
$A^\sim = G A^\ast  G $.
Let $A,B\in \mathbb{C}_{n,n}$.
It is easy to check that
$(AB)^{\sim}=B^{\sim}A^{\sim}$
and
$(A^{\sim})^{\sim}=A$.

The Minkowski inverse of
a matrix $A\in \mathbb{C}_{n,n}$
in $\mathcal{M}$
is defined
as the  matrix $X\in \mathbb{C}_{n,n}$
satisfying the following four conditions,
\cite{Meenakshi2000pnsaa1}:
\begin{align}
\label{WWL-MCE-lemma2.3}
(1)~AXA=A,
(2)~XAX=X,
(3^{\mathfrak{m}})~(AX)^\sim=AX,
(4^{\mathfrak{m}})~(XA)^\sim=XA,
\end{align}
 {{ where $X$
is denoted by $A^{\mathfrak{m}}$.}}
It is worthy to note that
the Minkowski inverse $A^{\mathfrak{m}}$
exists
if and only if
\begin{align}
\label{WWL-MCE-4-20191230-3}
{\rm rk}\left(A^\sim A\right)
 =
 {\rm rk}\left(A A^\sim\right)
 =
 {\rm rk}(A), \
 \mbox{\cite{Meenakshi2000pnsaa1}}.
\end{align}
If $A^{\mathfrak{m}}$ exists,
then it is unique,
\cite{Meenakshi2000pnsaa1,Petrovic2015amc157,Renardy1996laa53}.

In  \cite{Wang2019laa299},
Wang, Li and Liu
defined the ${\mathfrak{m}}$-core inverse
in  $\mathcal{M}$:
Let $A\in\mathbb{C}^{\tiny{\mbox{\rm CM}}}_n$,
if there exists $X\in\mathbb{C}_{n,n}$
satisfying the following three equations:
 \begin{align}
 \label{WWL-MCE-2}
(1)~AXA=A,
~(2^{l})~AX^{2}=X,
~(3^{m})~(AX)^{\sim}=AX,
\end{align}
then $X$ is called the ${\mathfrak{m}}$-core inverse of $A$,
 and is denoted by $A^{\tiny\textcircled{m}}$.
If $A^{\tiny\textcircled{m}}$ exists, then it is unique.
For $A\in\mathbb{C}^{\tiny{\mbox{\rm CM}}}_n $,
$A$  is {${\mathfrak{m}}$}-core invertible
if and only if
\begin{align}
\label{MC-Lemma-1.1}
{\rm rk}(A^{\sim}A)
=
 {\rm rk}(A).
\end{align}

\bigskip

 In \cite{Wang2016laa289},
 Wang introduced the core-EP decomposition:
 Let $A\in\mathbb{C}_{n,n}$
with
${\rm rk}(A^{k})=r$
and
${\mbox{\rm Ind}}(A)=k$.
Then
\begin{align}
\label{WWL-MCE-5}
A=A_1+A_2
\end{align}
where
$A_{1}\in\mathbb{C}^{\tiny{\mbox{\rm CM}}}_{n}$,
$A^{k}_{2}=0$
and
$A^{*}_{1}A_{2}=A_{2}A_{1}=0$.
Here one or both of $A_{1}$ and $A_{2}$ can be null.
Furthermore,
there exists an $n\times n$ unitary matrix $U$
such that
 \begin{align}
\label{WWL-MCE-666}
A_{1}=
U
\left[
\begin{matrix}
 T &  S\\
 0 &  0
\end{matrix}
\right]
U^{*}
\mbox{ \  and \ }
A_{2}=
U
\left[
\begin{matrix}
 0 &  0\\
   0          &  N
\end{matrix}
\right]
U^{*},
\end{align}
where
{{$S\in\mathbb{C}_{r, n-r}$,
$T\in\mathbb{C}_{r, r}$ is invertible,
$N\in\mathbb{C}_{n-r, n-r}$ is nilpotent,
and $N^k=0$.}}

When  $A\in\mathbb{C}^{\tiny{\mbox{\rm CM}}}_n$,
it is obvious that $N=0$, that is,  $A=A_1$.
Then the core inverse $A^{\tiny\textcircled{\#}}$  of $A$ is
\begin{align}
A^{\tiny\textcircled{\#}}
&
=
U
\left[
\begin{matrix}
 T^{-1} &  0\\
 0 &  0
\end{matrix}
\right]
U^{*},
\end{align}
where $A^{\tiny\textcircled{\#}}$ denotes the core inverse of $A$,
and the core inverse
$A^{\tiny\textcircled{\#}}$ of $A$
is the unique solution of
$AXA=A$,
$AX^{2}=X$, and $(AX)^\ast=AX$.

Denote
\begin{align}
\label{MC-8}
{U}^\ast G {U}
=
\left[
\begin{matrix}
  \widehat{G}_1   & \widehat{ G}_2  \\
  \widehat{G}_3   &  \widehat{G}_4
\end{matrix}
\right],
\end{align}
in which $\widehat{G}_1\in\mathbb{C}_{r, r}$.

When $A\in\mathbb{C}^{\tiny{\mbox{\rm CM}}}_n$ with
${\rm rk}(A^{\sim}A)
=
 {\rm rk}(A)$,
by applying (\ref{WWL-MCE-5}), (\ref{WWL-MCE-666})
and
(\ref{MC-8}),
we have
\begin{align}
\nonumber
{\rm rk}\left(A^\sim A\right)
&
=
{\rm rk}\left(GA^*GA\right)
=
{\rm rk}\left(GU
\left[
\begin{matrix}
 T^* &  0\\
 S^* &  0
\end{matrix}
\right]
U^{*}GU
\left[
\begin{matrix}
 T &  S\\
 0 &  0
\end{matrix}
\right]
U^{*}\right)
\\
&
\nonumber
=
{\rm rk}\left(
\left[
\begin{matrix}
 T^* &  0\\
 S^* &  0
\end{matrix}
\right]
\left[
\begin{matrix}
  \widehat{G}_1   & \widehat{ G}_2  \\
  \widehat{G}_3   &  \widehat{G}_4
\end{matrix}
\right]
\left[
\begin{matrix}
 T &  S\\
 0 &  0
\end{matrix}
\right]\right)
\\
&
\nonumber
=
{\rm rk}\left(
\left[
\begin{matrix}
 T^* \\
 S^*
\end{matrix}
\right]
\left[
\begin{matrix}
  \widehat{G}_1   & \widehat{ G}_2  \\
\end{matrix}
\right]
\left[
\begin{matrix}
 T &  S\\
 0 &  0
\end{matrix}
\right]\right)
=
{\rm rk}\left(
\left[
\begin{matrix}
 T^* \\
 S^*
\end{matrix}
\right]
\widehat{G}_1
\left[
\begin{matrix}
 T &  S\\
\end{matrix}
\right]\right)
\\
&
\nonumber
\leq {\rm rk} \left(\widehat{G}_1\right)
\leq r.
\end{align}
When ${\rm rk}\left(A^\sim A\right)=r$,
it follows that
${\rm rk} \left(\widehat{G}_1\right)=r$,
that is,
$\widehat{G}_1$ is invertible.
On the contrary,
suppose that
$\widehat{G}_1$ is invertible,
by applying
$
r={\rm rk}\left(A\right)\geq
{\rm rk}\left(A^\sim A\right)\geq
{\rm rk}\left(T^*
\widehat{G}_1
 T\right)=r$,
 we have $
{\rm rk}\left(A^\sim A\right)=r$.
Therefore,
we conclude that ${\rm rk}(A)={\rm rk}\left(A^\sim A\right)$
if and only if
$\widehat{G}_1$ is invertible.
It follows
by using (\ref{MC-Lemma-1.1})
 that
$A$  is {${\mathfrak{m}}$}-core invertible
if
and
only
if
$\widehat{G}_1$ is invertible.

Furthermore,
the {${\mathfrak{m}}$}-core inverse of $A$ can be expressed as form
\begin{align}
 \label{MC-211-1}
\mathcal{X}
=
 {U}\left[ {{\begin{matrix}
T^{-1}\widehat{G}_{1}^{-1}   & 0   \\
 0   & 0   \\
\end{matrix} }} \right]{U}^\ast G.
\end{align}
We give the outline of the proof as follows.

Since $G^*=G$,
we have
\begin{align}\label{MC-21111-1}
\widehat{G}_{1}^*=\widehat{G}_{1}.
\end{align}
By (\ref{WWL-MCE-5}), (\ref{WWL-MCE-666}),
(\ref{MC-8}),
(\ref{MC-211-1}) and (\ref{MC-21111-1}),
we have
\begin{align}
\nonumber
A\mathcal{X}A
&
=U
\left[
\begin{matrix}
 T &  S\\
 0 &  0
\end{matrix}
\right]
U^{*} {U}\left[ {{\begin{matrix}
T^{-1}\widehat{G}_{1}^{-1}   & 0   \\
 0   & 0   \\
\end{matrix} }} \right]{U}^\ast GU
\left[
\begin{matrix}
 T &  S\\
 0 &  0
\end{matrix}
\right]
U^{*}
\\
\label{20201109-1}
&
=U
\left[ {{\begin{matrix}
\widehat{G}_{1}^{-1}   & 0   \\
 0   & 0   \\
\end{matrix} }} \right]\left[
\begin{matrix}
  \widehat{G_1}   & \widehat{ G_2}  \\
  \widehat{G_3}   &  \widehat{G_4}
\end{matrix}
\right]
\left[
\begin{matrix}
 T &  S\\
 0 &  0
\end{matrix}
\right]
U^{*}
=U
\left[
\begin{matrix}
 T &  S\\
 0 &  0
\end{matrix}
\right]
U^{*}=A
,
\\
\nonumber
A\mathcal{X}^{2}
&
=U
\left[
\begin{matrix}
 T &  S\\
 0 &  0
\end{matrix}
\right]
U^{*} {U}\left[ {{\begin{matrix}
T^{-1}\widehat{G}_{1}^{-1}   & 0   \\
 0   & 0   \\
\end{matrix} }} \right]{U}^\ast G{U}\left[ {{\begin{matrix}
T^{-1}\widehat{G}_{1}^{-1}   & 0   \\
 0   & 0   \\
\end{matrix} }} \right]{U}^\ast G
\\
\nonumber
&
=U
\left[ {{\begin{matrix}
\widehat{G}_{1}^{-1}   & 0   \\
 0   & 0   \\
\end{matrix} }} \right]\left[
\begin{matrix}
  \widehat{G_1}   & \widehat{ G_2}  \\
  \widehat{G_3}   &  \widehat{G_4}
\end{matrix}
\right]\left[ {{\begin{matrix}
T^{-1}\widehat{G}_{1}^{-1}   & 0   \\
 0   & 0   \\
\end{matrix} }} \right]{U}^\ast G
\nonumber
\\
\label{20201109-2}
&
=U\left[ {{\begin{matrix}
T^{-1}\widehat{G}_{1}^{-1}   & 0   \\
 0   & 0   \\
\end{matrix} }} \right]{U}^\ast G=\mathcal{X},
\\
\nonumber
\left(A\mathcal{X}\right)^{\sim}
&
=\left(U
\left[
\begin{matrix}
 T &  S\\
 0 &  0
\end{matrix}
\right]
U^{*} {U}\left[ {{\begin{matrix}
T^{-1}\widehat{G}_{1}^{-1}   & 0   \\
 0   & 0   \\
\end{matrix} }} \right]{U}^\ast G\right)^{\sim}
=\left(U
\left[ {{\begin{matrix}
\widehat{G}_{1}^{-1}   & 0   \\
 0   & 0   \\
\end{matrix} }} \right]U^*G\right)^{\sim}
\\
\label{20201109-3}
&
=U
\left[ {{\begin{matrix}
\left(\widehat{G}_{1}^{-1}\right)^*   & 0   \\
 0   & 0   \\
\end{matrix} }} \right]U^*G
=U
\left[ {{\begin{matrix}
\widehat{G}_{1}^{-1}   & 0   \\
 0   & 0   \\
\end{matrix} }} \right]U^*G
=A\mathcal{X}.
\end{align}
Therefore,
by applying (\ref{20201109-1}), (\ref{20201109-2}), (\ref{20201109-3}) and (\ref{WWL-MCE-2})
we conclude that
$\mathcal{X}$
is the {${\mathfrak{m}}$}-core inverse of $A$, and
\begin{align}
 \label{20201109-4}
A^{\tiny\textcircled{m}}
=
 {U}\left[ {{\begin{matrix}
T^{-1}\widehat{G}_{1}^{-1}   & 0   \\
 0   & 0   \\
\end{matrix} }} \right]{U}^\ast G.
\end{align}

{{In recent years,
generalized core inverse has been extensively studied by many researchers,
particularly to the core-EP inverse.}}
The core-EP inverse of $A$
is the unique {{matrix \cite{Prasad2014lma792}}},
which satisfies
\begin{align}
(1^k)XA^{k+1}=A^{k},
(2)XAX=X,
(3)(AX)^{*}=AX,
\left(4^r\right)\mathcal{R}(X)\subseteq \mathcal{R}(A^{k}),
\end{align}
and is denoted by $A^{\tiny\textcircled{\dag}} $,
where ${\mbox{\rm Ind}}(A)=k$.

In \cite{Prasad2014lma7922}, Ma and Stanimirovi\'{c} studied the
characterizations, approximation and perturbations of the core-EP inverse, and applied SMS algorithm for computing the
core-EP inverse.
In \cite{Prasad2014lma7925}, Ferreyra, Levis and Thome proposed generalized the core-EP inverse to rectangular matrices and gave properties of the weighted core-EP inverse.
In \cite{Prasad2014lma7923},
Zhou et al proposed
three limit representations of the core-EP inverse.
In \cite{Prasad2014lma7929}, Mosi\'{c} studied the weighted core-EP inverse of an operator between Hilbert space.
In \cite{Prasad2014lma7926}, Ji and Wei applied the core-EP, weighted core-EP inverse of matrices to study constrained systems of linear equations.
More details of the core-EP inverse  and its applications  can be seen in
\cite{Prasad2014lma7928,Gao2018ca38,Prasad2014lma7922,Prasad2014lma7927,
Prasad2014lma79281,Sahoocam2020cam168}.

In this paper,
we introduce a generalization of the {${\mathfrak{m}}$}-core inverse
for square matrices of an arbitrary index.
We also give some of its characterizations,
properties and applications.

\section{The m-core-EP inverse in Minkowski space }

Let
$A\in\mathbb{C}_{n, n}$
with
${\mbox{\rm Ind}}(A)=k$,
 $A=A_1+A_2$ be of the core-EP decomposition of $A$,
and
$A_1$ and $A_2$ be as in  (\ref{WWL-MCE-666}).
Then
$T\in\mathbb{C}_{r, r}$ is invertible
and
\begin{align}
\label{WWL-MCE-6}
A^{k}=
U
\left[
\begin{matrix}
 T^{k} &  \widehat{T}\\
   0          &  0
\end{matrix}
\right]
U^{*},
\end{align}
and
\begin{align}
\label{WWL-MCE-7}
A^{k+1}=
U
\left[
\begin{matrix}
 T^{k+1} &  \overline{T}\\
   0          &  0
\end{matrix}
\right]
U^{*},
\end{align}
where
$\widehat{T}=T^{k-1}S+T^{k-2}SN+\cdots+TSN^{k-2}+SN^{k-1}$,
and
$\overline{T}=T^{k}S+T^{k-1}SN+\cdots+TSN^{k-1}+SN^{k}
=T^{k}S+T^{k-1}SN+\cdots+TSN^{k-1}$.
It is obvious that
$T^{-1}\overline{T}=\widehat{T}$.

When
$A\in\mathbb{C}^{\tiny{\mbox{\rm CM}}}_n$,
it is easy to check that
\begin{align}
\label{WWL-MCE-lemma2.6}
A^{\sharp}
&
=
 {U}
\left[
\begin{matrix}
   T^{-1} &    T^{-2}S  \\
       0     &       0
\end{matrix}
\right]
{U}^{*}.
\end{align}

\begin{lemma}
\label{WWL-MCE-734}
Let
$A\in\mathbb{C}_{n, n}$
with
${\mbox{\rm Ind}}(A)=k$,
then
${\mbox{\rm rk}}\left(A^{k}\right)
=
{\mbox{\rm rk}}\left(\left(A^{k}\right)^{\sim}A^{k}\right)$
if and only if
$G_{1}$ is invertible,
in which
$G_1\in\mathbb{C}_{{\mbox{\rm rk}}\left(A^{k}\right),
{\mbox{\rm rk}}\left(A^{k}\right)}$
is given as in
\begin{align}
 \label{WWL-MCE-4}
U^{*}GU
=
 \left[{\begin{matrix}
 G_{1}
    & G_{2}   \\
 G_{3}   & G_{4}   \\
\end{matrix} }\right],
\end{align}
{{ where $U$ is defined as in (\ref{WWL-MCE-666}).}}
\end{lemma}

\begin{proof}
Let
$A\in\mathbb{C}_{n, n}$
 be as in (\ref{WWL-MCE-5}) and (\ref{WWL-MCE-666}),
 $A^{k}$ be as in (\ref{WWL-MCE-6}),
and ${U}^\ast G {U}$ is given as in (\ref{WWL-MCE-4}).
Then
\begin{align}
\nonumber
{\rm rk}\left((A^{k})^\sim A^{k}\right)
&
=
{\rm rk}\left(G(A^{k})^{*}GA^{k}\right)
=
{\rm rk}\left(GU
\left[
\begin{matrix}
 (T^{k})^* &  0\\
 \widehat{T}^* &  0
\end{matrix}
\right]
U^{*}GU
\left[
\begin{matrix}
 T^{k} &  \widehat{T}\\
 0 &  0
\end{matrix}
\right]
U^{*}\right)
\\
&
\nonumber
=
{\rm rk}\left(
\left[
\begin{matrix}
 (T^{k})^* &  0\\
 \widehat{T}^* &  0
\end{matrix}
\right]
\left[
\begin{matrix}
 G_1   & G_2  \\
  G_3   &  G_4
\end{matrix}
\right]
\left[
\begin{matrix}
 T^{k} &  \widehat{T}\\
 0 &  0
\end{matrix}
\right]\right)
\\
&
\nonumber
=
{\rm rk}\left(
\left[
\begin{matrix}
 (T^{k})^* \\
 \widehat{T}^*
\end{matrix}
\right]
\left[
\begin{matrix}
  G_1   & G_2  \\
\end{matrix}
\right]
\left[
\begin{matrix}
 T^{k} &  \widehat{T}\\
 0 &  0
\end{matrix}
\right]\right)
=
{\rm rk}\left(
\left[
\begin{matrix}
 (T^{k})^* \\
 \widehat{T}^*
\end{matrix}
\right]
G_1
\left[
\begin{matrix}
 T^{k} &  \widehat{T}\\
\end{matrix}
\right]\right)
\\
&
\nonumber
\leq {\rm rk} \left(G_1\right)
\leq r.
\end{align}
When ${\rm rk}\left((A^{k})^\sim A^{k}\right)=r$,
it follows that
${\rm rk} \left({G}_1\right)=r$,
that is,
${G}_1$ is invertible.
On the contrary,
suppose that
${G}_1$ is invertible,
by applying
$
r={\rm rk}\left(A^{k}\right)\geq
{\rm rk}\left((A^{k})^\sim A^{k}\right)\geq
{\rm rk}\left((T^{k})^*
G_1
 T^{k}\right)=r$,
 we have $
{\rm rk}\left((A^{k})^\sim A^{k}\right)=r$.
Therefore,
we conclude that ${\rm rk}(A^{k})={\rm rk}\left((A^{k})^\sim A^{k}\right)$
if and only if
${G}_1$ is invertible.
\end{proof}

\begin{remark}
\label{WZHBAP-MT1-Remark}
Let
$A\in\mathbb{C}_{n, n}$
with
${\mbox{\rm Ind}}(A)=k$
and
${\mbox{\rm rk}}\left(A^{k}\right)
=
{\mbox{\rm rk}}\left(\left(A^{k}\right)^{\sim}A^{k}\right)$.
According to the
$(U^{*}GU)^{*}=U^{*}GU$,
we obtain
\begin{align*}
G_{1}^{*}=G_{1}
\mbox{\  and \ }
 G_{3}^{*}=G_{2}.
\end{align*}
\end{remark}

Let $A\in\mathbb{C}_{n, n}$ with
${\mbox{\rm Ind}}(A)=k$, and consider %the system of equations
\begin{align}\label{WZHBAP-MT11-Remark}
(1)XAX=X,
(2^k)XA^{k+1}=A^{k},
\left(3^{\mathfrak{m}}\right)(AX)^\sim=AX,
\left(4^r\right)\mathcal{R}(X)\subseteq \mathcal{R}\left(A^{k}\right).
\end{align}

\begin{theorem}
Let $A\in\mathbb{C}_{n,n}$
with ${\mbox{\rm Ind}}(A)=k$.
If there exists a solution of the four matrix equations $(1)$, $(2^{k})$, $(3^{m})$ and $(4^{r})$,
then the solution is unique.
\end{theorem}
\begin{proof}
Let
$A\in\mathbb{C}_{n, n}$
 be as given in (\ref{WWL-MCE-5}) and (\ref{WWL-MCE-666}).
We suppose that both
$X_{1}$ and $Y_{1}$ satisfy  (\ref{WZHBAP-MT11-Remark}). {{Let $U$ be defined as in (\ref{WWL-MCE-666}).}}
Denote
\begin{align}
\label{WWL-MCE-21}
X_{1}=
U
\left[
\begin{matrix}
 X_{11} &  X_{12}\\
   X_{21}     & {{ X_{22}}}
\end{matrix}
\right]
U^{*}
\mbox{ \ and \ }
Y_{1}=
U
\left[
\begin{matrix}
 Y_{11} &  Y_{12}\\
   Y_{21}     &  Y_{22}
\end{matrix}
\right]
U^{*}.
\end{align}
Next, we prove $X_{1}=Y_{1}$.

{{ From}}
$X_{1}A^{k+1}=A^{k}$
  and
$Y_{1}A^{k+1}=A^{k}$,
  we have
\begin{align}\label{WWL-MCE-20}
X_{1}A^{k+1}=Y_{1}A^{k+1}.
\end{align}
By applying (\ref{WWL-MCE-21}) and (\ref{WWL-MCE-20}), we have
\begin{align}
\nonumber
U
\left[
\begin{matrix}
   X_{11}  &      X_{12}  \\
       X_{21}   &       X_{22}
\end{matrix}
\right]
U^{*}U
\left[
\begin{matrix}
   T^{k+1}  &      \overline{T}  \\
       0   &       0
\end{matrix}
\right]
U^{*}
\nonumber
&
=
U
\left[
\begin{matrix}
   Y_{11} &      Y_{12} \\
       Y_{21}     &       Y_{22}
\end{matrix}
\right]
U^{*}U
\left[
\begin{matrix}
   T^{k+1}  &      \overline{T}  \\
       0   &       0
\end{matrix}
\right]
U^{*},
\end{align}
that is,
\begin{align}
\nonumber
\left[
\begin{matrix}
   X_{11}T^{k+1}       &  X_{11}\overline{T}  \\
       X_{21}T^{k+1}   &  X_{21}\overline{T}
\end{matrix}
\right]
\nonumber
&
=
\left[
\begin{matrix}
   Y_{11}T^{k+1} &      Y_{11}\overline{T} \\
       Y_{21}T^{k+1}     &       Y_{21}\overline{T}
\end{matrix}
\right].
\end{align}
Since $T$ is invertible,
we get
\begin{align}
X_{11}=Y_{11}
\mbox{ \  and \ }
X_{21}=Y_{21}.
\end{align}

Applying
$\mathcal{R}(X_{1})\subseteq \mathcal{R}(A^{k})$
gives that
there exists a matrix $Z$ such that
\begin{align}\label{WWL-MCE-222}
X_{1}=A^{k}Z.
\end{align}
Write
\begin{align}
\label{WWL-MCE-2111}
Z=
U
\left[
\begin{matrix}
 Z_{11} &  Z_{12}\\
   Z_{21}     &  Z_{22}
\end{matrix}
\right]
U^{*}.
\end{align}
Substituting
(\ref{WWL-MCE-21}) and (\ref{WWL-MCE-2111})
into (\ref{WWL-MCE-222}),
we have
\begin{align}
\nonumber
\left[
\begin{matrix}
   X_{11}  &      X_{12}  \\
       X_{21}   &       X_{22}
\end{matrix}
\right]
\nonumber
&
=
\left[
\begin{matrix}
   T^{k}Z_{11}+\widehat{T}Z_{21} &    T^{k}Z_{12}+\widehat{T}Z_{22}   \\
       0     &       0
\end{matrix}
\right],
\end{align}
so $X_{21}=0$ and $X_{22}=0$.
Similarly,
we get $Y_{21}=0$ and $Y_{22}=0$.

Since $(AX_{1})^{\thicksim}=AX_{1}$,
\begin{align}
\label{WWL-MCE-223}
GX_{1}^{*}A^{*}G=AX_{1}.
\end{align}
By applying
{{ (\ref{WWL-MCE-666}),}}
(\ref{WWL-MCE-4}),
($\ref{WWL-MCE-21}$),
(\ref{WWL-MCE-223}),
$X_{21}=0$ and $X_{22}=0$,
we have
\begin{align}
\nonumber
&
0
=
GX_{1}^{*}A^{*}G-AX_{1}
\\
\nonumber
&
=
GU
\left[
\begin{matrix}
   X_{11}^{*}  &      {{  0}}  \\
       X_{12}^{*}   &       {{  0}}
\end{matrix}
\right]
\left[
\begin{matrix}
   T^{*}  &      0  \\
       S^{*}   &       N^{*}
\end{matrix}
\right]
U^{*}G
-
U
\left[
\begin{matrix}
   T &      S \\
       0     &       N
\end{matrix}
\right]
\left[
\begin{matrix}
   X_{11}  &      X_{12}  \\
        {{  0}}  &      {{  0}}
\end{matrix}
\right]
U^{*}
\\
\nonumber
&
=
GU \left(\left[
\begin{matrix}
   X_{11}^{*}T^{*}G_{1}  &     X_{11}^{*}T^{*}G_{2}   \\
       X_{12}^{*}T^{*}G_{1}   &       X_{12}^{*}T^{*}G_{2}
\end{matrix}
\right]
-
\left[
\begin{matrix}
   G_{1}TX_{11}  &      G_{1}TX_{12}  \\
       G_{3}TX_{11}   &       G_{3}TX_{12}
\end{matrix}
\right]\right)U^\ast.
\end{align}
Then
$X_{11}^{*}T^{*}G_{2}=G_{1}TX_{12}$,
that is,
$X_{12}=T^{-1}G_{1}^{-1}X_{11}^{*}T^{*}G_{2}$.
Similarly,
we get
$Y_{12}=T^{-1}G_{1}^{-1}Y_{11}^{*}T^{*}G_{2}$.
{{Applying $X_{11}=Y_{11}$, we have $X_{12}=Y_{12}.$}}
Therefore, we get $X_{1}=Y_{1}$.
\end{proof}

\begin{definition}
\label{WWL-MCE-df-3.1}
Let
$A\in\mathbb{C}_{n, n}$
with
${\mbox{\rm Ind}}(A)=k$.
If there exists $X\in \mathbb{C}_{n,n}$
satisfying the system of equations (\ref{WZHBAP-MT11-Remark}),
then
$X$ is called the  ${\mathfrak{m}}$-core-EP inverse of $A$
in $\mathcal{M}$,
and is denoted by $A^{{\tiny{\textcircled{E}}}}$.
\end{definition}

In the following theorems of this section,
by using
matrix decomposition,
Drazin inverse,
group inverse,
{${\mathfrak{m}}$}-core inverse
and
Minkowski inverse,
we derive several characterizations of
the {${\mathfrak{m}}$}-core-EP inverse.

\begin{theorem}
\label{WWL-MCE-TH3.2}
Let $A\in\mathbb{C}_{n, n}$ with ${\mbox{\rm Ind}}(A)=k$.
Then $A$ is {${\mathfrak{m}}$}-core-EP invertible
if and only if
\begin{align}
\label{WWL-MCE-th3.2-2}
{\mbox{\rm rk}}\left(A^{k}\right)
=
{\mbox{\rm rk}}\left(\left(A^{k}\right)^{\sim}A^{k}\right).
\end{align}
Furthermore,
let the decomposition of $A$
be as in (\ref{WWL-MCE-5}),
then $A^{{\tiny{\textcircled{E}}}}$
has the decomposition in the form
\begin{align}
\label{WWL-MCE-th3.2}
A^{{\tiny{\textcircled{E}}}}
=
U
\left[
\begin{matrix}
 T^{-1}G_{1}^{-1} &  0\\
   0     &  0
\end{matrix}
\right]
U^{*}G,
\end{align}
where $G_{1}$ is of the form (\ref{WWL-MCE-4}).
\end{theorem}

\begin{proof}
``$\Rightarrow$''
From the equation $(2^k)$, we have
$\mathcal{R}(A^{k})\subseteq\mathcal{R}(X)$.
It follows from the equation $(4^{r})$ that $\mathcal{R}(A^{k})=\mathcal{R}(X)$.
Therefore, there exists a nonsingular $P$ such that $X=A^k P$.

{{ From}}
$XAX
=
XAA^kP
=
X\left(A^kP\right)P^{-1}AP
=
X^2P^{-1}AP=X$,
we get
${\mbox{\rm rk}}(X)
=
{\mbox{\rm rk}}(X^2P^{-1}AP)
\leq
{\mbox{\rm rk}}(X^2)
\leq
{\mbox{\rm rk}}(X)$,
that is,
\begin{align}
\label{WWL-MCE-2-20191230}
{\mbox{\rm rk}}(X)
=
{\mbox{\rm rk}}(X^2)
=
{\mbox{\rm rk}}(A^k).
\end{align}
Applying
$XAX=X$
and
$(AX)^\sim=AX$,
we get
$
X^2
=
XAX A^kP
=
XP^\sim  A^\sim \left(A^k\right)^\sim A^kP$,
then
{{ ${\mbox{\rm rk}}\left(X^2\right)
\leq
{\mbox{\rm rk}}\left(\left(A^{k}\right)^{\sim}A^{k}\right)
\leq
{\mbox{\rm rk}}\left(A^{k}\right)$.}}

It follows from (\ref{WWL-MCE-2-20191230})
that
we get (\ref{WWL-MCE-th3.2-2}).

``$\Leftarrow$''
Let ${\mbox{\rm rk}}\left(A^{k}\right)
=
{\mbox{\rm rk}}\left(\left(A^{k}\right)^{\sim}A^{k}\right)$.
Applying Lemma \ref{WWL-MCE-734},
we get that $G_{1}$ is invertible.
Write
\begin{align}
\label{WWL-MCE-8}
X=
U
\left[
\begin{matrix}
 T^{-1}G_{1}^{-1} &  0\\
   0     &  0
\end{matrix}
\right]
U^{*}G,
\end{align}
where
$T\in\mathbb{C}_{r, r}$
is given in (\ref{WWL-MCE-666}).

By applying (\ref{WWL-MCE-4}) and (\ref{WWL-MCE-8}), we have
\begin{align}\nonumber
XAX
&=
U
\left[
\begin{matrix}
   T ^{-1}G^{-1}_{1} &      0  \\
       0   &       0
\end{matrix}
\right]
U^{*}GU
\left[
\begin{matrix}
   T &       S \\
       0     &       N
\end{matrix}
\right]
U^{*}U
\left[
\begin{matrix}
   T^{-1}G^{-1}_{1} &       0 \\
       0     &       0
\end{matrix}
\right]
U^{*}G
\\
\nonumber
&
=
U
\left[
\begin{matrix}
  T^{-1}G^{-1}_{1}  &   0 \\
      0     &       0
\end{matrix}
\right]
\left[
\begin{matrix}
 G_{1} &     G_{2} \\
       G_{3}     &       G_{4}
\end{matrix}
\right]
\left[
\begin{matrix}
   T &       S \\
       0     &       N
\end{matrix}
\right]
\left[
\begin{matrix}
  T^{-1}G^{-1}_{1}  &   0 \\
       0     &       0
\end{matrix}
\right]
U^{*}G
\\
\nonumber
&
=
U
\left[
\begin{matrix}
  T^{-1}G_{1}^{-1}   &  0  \\
       0     &      0
\end{matrix}
\right]
U^{*}G=X.
\end{align}
By applying (\ref{WWL-MCE-7}) and (\ref{WWL-MCE-8}), then
\begin{align}\nonumber
XA^{k+1}
&=
U
\left[
\begin{matrix}
   T ^{-1}G^{-1}_{1} &      0  \\
       0   &       0
\end{matrix}
\right]
U^{*}GU
\left[
\begin{matrix}
   T^{k+1} &       \overline{T} \\
       0     &       0
\end{matrix}
\right]
U^{*}
\\
\nonumber
&
=
U
\left[
\begin{matrix}
  T^{-1}G^{-1}_{1}  &   0 \\
      0     &       0
\end{matrix}
\right]
\left[
\begin{matrix}
 G_{1} &     G_{2} \\
       G_{3}     &       G_{4}
\end{matrix}
\right]
\left[
\begin{matrix}
  T^{k+1}  &   \overline{T} \\
       0     &       0
\end{matrix}
\right]
U^{*}
\\
\nonumber
&
=
U
\left[
\begin{matrix}
  T^{-1}   &  T^{-1}G_{1}^{-1}G_{2}  \\
       0     &      0
\end{matrix}
\right]
\left[
\begin{matrix}
   T^{k+1} &       \overline{T} \\
       0     &       0
\end{matrix}
\right]
U^{*}
\\
\nonumber
&
=
U
\left[
\begin{matrix}
  T^{k}   &  T^{-1}\overline{T}  \\
       0     &      0
\end{matrix}
\right]
U^{*}=
U
\left[
\begin{matrix}
  T^{k}   &  \widehat{T}  \\
       0     &      0
\end{matrix}
\right]
U^{*}=A^{k}.
\end{align}
By applying (\ref{WWL-MCE-4}) and (\ref{WWL-MCE-8}), we have
\begin{align}\nonumber
AX
=
U
\left[
\begin{matrix}
 T  &  S  \\
       0     &      N
\end{matrix}
\right]
\left[
\begin{matrix}
   T^{-1}G^{-1}_{1} &       0 \\
       0     &       0
\end{matrix}
\right]{{ U^{*}G}}
=
U
\left[
\begin{matrix}
  G^{-1}_{1}   &  0  \\
       0     &      0
\end{matrix}
\right]
U^{*}G.
\end{align}
Since
\begin{align}
\nonumber
(AX)^\sim
&=G(AX)^{*}G
=
GG^{*}U
\left[
\begin{matrix}
   (G_{1}^{-1})^{*}  &      0  \\
       0   &       0
\end{matrix}
\right]
U^{*}G
=
U
\left[
\begin{matrix}
  (G_{1}^{-1})^{*}  &   0 \\
      0     &       0
\end{matrix}
\right]
U^{*}G,
\end{align}
 applying Remark \ref{WZHBAP-MT1-Remark},
 we get $(AX)^\sim=AX$.

Let
\begin{align}
\label{WWL-MCE-X}
\overline{X}
&
=
U
\left[
\begin{matrix}
   (T^{k+1})^{-1} &     (T^{k+1})^{-1}G_{1}^{-1}G_{2}  \\
       0     &       0
\end{matrix}
\right]
U^{*}.
\end{align}
It follows
from applying
(\ref{WWL-MCE-6}),
 (\ref{WWL-MCE-8}), (\ref{WWL-MCE-X})
 and
\begin{align}
\nonumber
A^{k} \overline{X}
&
=
A^{k}U
\left[
\begin{matrix}
   (T^{k+1})^{-1} &     (T^{k+1})^{-1}G_{1}^{-1}G_{2}  \\
       0     &       0
\end{matrix}
\right]
U^{*}
=
U
\left[
\begin{matrix}
 T^{-1}G_{1}^{-1} &  0\\
   0     &  0
\end{matrix}
\right]
U^{*}G=X,
\end{align}
that  $\mathcal{R}(X)\subseteq \mathcal{R}(A^{k})$.

Hence, $A$ is {${\mathfrak{m}}$}-core-EP invertible,
and
$A^{{\tiny{\textcircled{E}}}}$ is of the  form (\ref{WWL-MCE-th3.2}).
\end{proof}

It is  well known that every complex matrix is core-EP invertible.
But,
from (\ref{WWL-MCE-th3.2-2})
we see that
it is  untrue to say that all complex matrices are {${\mathfrak{m}}$}-core-EP invertible.
Even when
one matrix  $A$ is core-EP invertible
and is
{${\mathfrak{m}}$}-core-EP invertible,
it is generally true that
the two generalized inverses are different.

\begin{example}
\label{WWL-MCE-example-3.1}
Let
$A
=
 \left[{\begin{matrix}
 \frac{3-\sqrt6+\sqrt3-\sqrt2}{6} &  \frac{\sqrt6+2\sqrt3-2\sqrt2}{6}
 & \frac{3+\sqrt6-\sqrt3+\sqrt2}{6}  \\
   \frac{1}{3\sqrt2} & \frac{2}{3\sqrt2} & -\frac{1}{3\sqrt2} \\
   \frac{3-\sqrt6+\sqrt3+\sqrt2}{6} &  \frac{\sqrt6+2\sqrt3+2\sqrt2}{6}
    & \frac{3+\sqrt6-\sqrt3-\sqrt2}{6}
\end{matrix} }\right]$
with
 ${\mbox{\rm Ind}}(A)=2$  and  ${\mbox{\rm rk}}(A^2)=1$.
There exists a unitary matrix
$$U=\left[
\begin{matrix}
 \frac{1}{\sqrt2} &  -\frac{1}{\sqrt3} & \frac{1}{\sqrt6}  \\
   0 & \frac{1}{\sqrt3} & \frac{2}{\sqrt6} \\
   \frac{1}{\sqrt2} &  \frac{1}{\sqrt3} & -\frac{1}{\sqrt6}
\end{matrix}
\right],$$
such that
\begin{align*}
A
=
U
\left[
\begin{matrix}
 1 &  1 & 1  \\
   0 & 0 &1 \\
   0 &  0 & 0
\end{matrix}
\right]
U^{*}.
\end{align*}
By applying (\ref{WWL-MCE-4}),
we have
\begin{align*}
U^{*}GU
=
\left[{\begin{matrix}
 G_{1}
    & G_{2}   \\
 G_{3}   & G_{4}   \\
\end{matrix} }\right]
=
\left[
\begin{matrix}
 0 &  -\frac{2}{\sqrt6} & \frac{1}{\sqrt3}  \\
   -\frac{2}{\sqrt6} & -\frac{1}{3} & -\frac{2}{3\sqrt2} \\
   \frac{1}{\sqrt3} & -\frac{2}{3\sqrt2}  & -\frac{2}{3}
\end{matrix}
\right],
\end{align*}
where $G_1\in\mathbb{C}_{1,1}$.
Since $G_{1}=0$ is singular,
by applying Lemma \ref{WWL-MCE-734} and
Theorem \ref{WWL-MCE-TH3.2},
{{ then $A^{\tiny\textcircled{E}}$ doesn't exist.}}
\end{example}

\begin{example}
Let
$A
=
 \left[{\begin{matrix}
 0 & \frac{4}{3}  & -\frac{1}{3}  \\
   -\frac{1}{3} & 1 & -\frac{1}{3} \\
   -\frac{2}{3} &  -\frac{2}{3} & 0
\end{matrix} }\right]$
with
 ${\mbox{\rm Ind}}(A)=2$ and ${\mbox{\rm rk}}(A^2)=1$.
There exists a unitary matrix
$$U=\left[
\begin{matrix}
 \frac{2}{3} &  \frac{1}{3} & -\frac{2}{3}  \\
   \frac{1}{3} & \frac{2}{3} & \frac{2}{3} \\
   -\frac{2}{3} &  \frac{2}{3} & -\frac{1}{3}
\end{matrix}
\right],$$
such that
\begin{align*}
A
=
U
\left[
\begin{matrix}
 1 &  1 & 1  \\
   0 & 0 & 1\\
   0 &  0 & 0
\end{matrix}
\right]
U^{*},
\end{align*}
$T=1$.
In \cite[Theorem 3.2]{Wang2016laa289}, we know that
\begin{align*}
A^{\tiny\textcircled{\dag}}
=
U
\left[
\begin{matrix}
 T^{-1} &  0  \\
   0 & 0
\end{matrix}
\right]
U^{*}.
\end{align*}
Then
\begin{align*}
A^{\tiny\textcircled{\dag}}
=
U
\left[
\begin{matrix}
 T^{-1} &  0 & 0  \\
   0 & 0 &0 \\
   0 &  0 & 0
\end{matrix}
\right]
U^{*}
=
\left[
\begin{matrix}
 \frac{4}{9} &  \frac{2}{9} & -\frac{4}{9}  \\
   \frac{2}{9} & \frac{1}{9} & -\frac{2}{9}\\
   -\frac{4}{9} &  -\frac{2}{9} & \frac{4}{9}
\end{matrix}
\right].
\end{align*}
By applying (\ref{WWL-MCE-4}), we have
\begin{align*}
U^{*}GU
=
\left[{\begin{matrix}
 G_{1}
    & G_{2}   \\
 G_{3}   & G_{4}   \\
\end{matrix} }\right]
=
\left[
\begin{matrix}
 -\frac{1}{9} &  \frac{4}{9} & -\frac{8}{9}  \\
   \frac{4}{9} & -\frac{7}{9} & -\frac{4}{9} \\
   -\frac{8}{9} &  -\frac{4}{9} & -\frac{1}{9}
\end{matrix}
\right],
\end{align*}
where $G_1\in\mathbb{C}_{1,1}$.
Since $G_{1}=-\frac{1}{9}$ is nonsingular,
by applying
Lemma \ref{WWL-MCE-734}
and
Theorem \ref{WWL-MCE-TH3.2},
then $A$ is {${\mathfrak{m}}$}-core-EP invertible.
Therefore,
\begin{align}
A^{\tiny\textcircled{E}}
\nonumber
&
=
U
\left[
\begin{matrix}
 T^{-1}G_{1}^{-1} &  0 & 0  \\
   0 & 0 & 0\\
   0 &  0 & 0
\end{matrix}
\right]
U^{*}G
=
\left[
\begin{matrix}
 -4 &  2 & -4  \\
   -2 & 1 &-2 \\
   4 &  -2 & 4
\end{matrix}
\right].
\end{align}
From above, we know that
$A^{\tiny\textcircled{\dag}}\neq A^{\tiny\textcircled{E}}$.
\end{example}

\begin{lemma}
[\cite{Ben2003book,Wang2018book}]
Let $A\in\mathbb{C}_{n,n}$
with ${\mbox{\rm Ind}}(A)=k$,
and the Drazin inverse of $A$ is defined as
the unique matrix $X\in\mathbb{C}_{n, n}$ satisfying the equations:
\begin{align}
(1^k)A^{k}XA=A^{k},\
(2)XAX=X,\
(3)AX=XA,
\end{align}
and $X$ is denoted by $A^{D}$.
In particular,
when $A\in\mathbb{C}^{\tiny{\mbox{\rm CM}}}_n$,
$X$ is called the group inverse of $A$,
and we denote $X=A^{\sharp}$.
\end{lemma}

\begin{lemma}
[\cite{Ben2003book,Wang2018book}]
Let $A\in\mathbb{C}_{n, n}$ with ${\mbox{\rm Ind}}(A)=k$,
then
\begin{align}
\label{WWL-MCE-9}
A^{D}=A^{k}(A^{k+1})^{\sharp}.
\end{align}
\end{lemma}

Let $A=A_1+A_2$ be of the core-EP decomposition of $A$,
and
$A_1$ and $A_2$ be as in  (\ref{WWL-MCE-666}),
by applying (\ref{WWL-MCE-7})
and
(\ref{WWL-MCE-lemma2.6}),
we have
\begin{align}
\label{WWL-MCE-12}
(A^{k+1})^{\sharp}=
U
\left[
\begin{matrix}
 (T^{k+1})^{-1} & (T^{k+1})^{-2}\overline{T}\\
   0     &  0
\end{matrix}
\right]
U^{*}.
\end{align}
By applying
(\ref{20201109-3}),
(\ref{WWL-MCE-6})
and
(\ref{WWL-MCE-12}),
we can check that
\begin{align}
\label{WWL-MCE-13}
A^{D}
&
=
U
\left[
\begin{matrix}
  T^{-1}  &   T^{-k-2}\overline{T} \\
      0     &       0
\end{matrix}
\right]
U^{*}
\ \mbox{ and } \
\left(A^{k}\right)^{\tiny\textcircled{m}}
=
U
\left[
\begin{matrix}
 \left(T^{k}\right)^{-1}G_{1}^{-1} & 0\\
   0     &  0
\end{matrix}
\right]
U^{*}G,
\end{align}
%where $T$ and $\overline{T}$ are as in (\ref{WWL-MCE-7}).
%
in which $G_{1}$ is given as in (\ref{WWL-MCE-4}).
Furthermore,
by applying (\ref{WWL-MCE-6}) and (\ref{WWL-MCE-13}),
we can obtain
\begin{align}\nonumber
A^{k}A^{D}\left(A^{k}\right)^{\tiny\textcircled{m}}
&=
U
\left[
\begin{matrix}
   T^{k}  &      \widehat{T}  \\
       0   &       0
\end{matrix}
\right]
\left[
\begin{matrix}
   T^{-1} &       T^{-k-2}\overline{T} \\
       0     &       0
\end{matrix}
\right]
\left[
\begin{matrix}
   \left(T^{k}\right)^{-1}G_{1}^{-1} &       0 \\
       0     &       0
\end{matrix}
\right]
U^{*}G
\\
\nonumber
&
=
U
\left[
\begin{matrix}
 T^{-1}G_{1}^{-1}  &  0  \\
       0     &      0
\end{matrix}
\right]
U^{*}G
=
A^{\tiny\textcircled{E}}.
\end{align}
Therefore,
we get one characterization of the {${\mathfrak{m}}$}-core-EP inverse.

\begin{theorem}
Let $A\in\mathbb{C}_{n,n}$
with ${\mbox{\rm Ind}}(A)=k$.
If $A$ is {${\mathfrak{m}}$}-core-EP invertible, then
\begin{align}\label{WWL-MCE-11}
A^{{\tiny{\textcircled{E}}}}
=
A^{k}A^{D}\left(A^{k}\right)^{\tiny\textcircled{m}}.
\end{align}
\end{theorem}

\begin{theorem}
\label{WWL-Theorem-MCE-2}
Let
$A\in\mathbb{C}_{n,n}$ with ${\mbox{\rm Ind}}(A)=k$.
If the  {${\mathfrak{m}}$}-core-EP inverse exists,
then we have
\begin{align}
\label{WWL-Theorem-MCE-EQ-1}
A^{{\tiny{\textcircled{E}}}}
=
A_{1}^{\tiny\textcircled{m}},
\end{align}
where $A_1$ is given as in (\ref{WWL-MCE-5}).
\end{theorem}
\begin{proof}
Let $A_1$ be as in (\ref{WWL-MCE-5}),
then applying (\ref{WWL-MCE-666}),
we get
\begin{align}
\nonumber
A_{1}^{\tiny\textcircled{m}}=
U
\left[
\begin{matrix}
 T^{-1}G_{1}^{-1} & 0\\
   0     &  0
\end{matrix}
\right]
U^{*}G,
\end{align}
where $G_{1}$ is given as in (\ref{WWL-MCE-4}).{{ By applying
(\ref{WWL-MCE-th3.2}),}}
we have (\ref{WWL-Theorem-MCE-EQ-1}).
\end{proof}

\begin{theorem}
\label{WWL-Theorem-MCE-3}
Let $A\in\mathbb{C}_{n,n}$
 with ${\mbox{\rm Ind}}(A)=k$
 and
 ${\mbox{\rm rk}}(A^{k})
 ={\mbox{\rm rk}}((A^{k})^{\sim}A^{k})
 ={\mbox{\rm rk}}(A_{1}A_{1}^{\sim})=r$.
Then
\begin{align}
\label{WWL-MCE-4-1}
A^{{\tiny{\textcircled{E}}}}=A_{1}^{\sharp}A_{1}A_{1}^{m}.
\end{align}
\end{theorem}

\begin{proof}
Let the core-EP decomposition of $A$ be as in (\ref{WWL-MCE-5}),
and
$ A_{1}$ and $A_2$ be as in (\ref{WWL-MCE-666}).
Then
  $A_1$ is group invertible,
and
$(A_{1}+A_{2})^{k}
= A_{1}^{k}+A_{1}^{k-1}A_{2}+A_{1}^{k-2}A_{2}^{2}
+\cdots+
A_{1}A_{2}^{k-1}+A_{2}^{k}$,
in which $k={\mbox{\rm Ind}}(A)$.

Applying
 ${\mbox{\rm rk}}\left(A^{k}\right)
 ={\mbox{\rm rk}}\left(\left(A^{k}\right)^{\sim}A^{k}\right)
 ={\mbox{\rm rk}}\left(A_{1}A_{1}^{\sim}\right)=r$,
 we get
  $
{\mbox{\rm rk}}\left(A_1\right)
=
{\mbox{\rm rk}}\left( A_1  A_1^{\sim}\right)
=
{\mbox{\rm rk}}\left( A_1 ^{\sim} A_1\right)$.
It follows from  (\ref{WWL-MCE-4-20191230-3}) that
$ A_{1}$ is Minkowski invertible.

Write
\begin{align*}
\mathfrak{X}
=
A_{1}^{\sharp}A_{1}A_{1}^{{\mathfrak{m}}}.
\end{align*}
Applying
$A_{2}A_{1}=0$,
$A_{1}^{\sharp}A_{1}=A_{1}A_1^{\sharp}$,
$A_{2}^{k}=0$,
$A_{1}^{\sharp}A_{1}A_{1}=A_1$
and
$A_{1}A_{1}^{m} A_{1}= A_{1}$,
we have
\begin{align}
\nonumber
\mathfrak{X}A^{k+1}
&=
A_{1}^{\sharp}A_{1}A_{1}^{m}(A_{1}+A_{2})^{k+1}
\\
\nonumber
&
=
A_{1}^{\sharp}A_{1}A_{1}^{m}(A_{1}^{k+1}+A_{1}^{k}A_{2}
+A_{1}^{k-1}A_{2}^{2}+\cdots+A_{2}^{k+1})
\\
\nonumber
&
=
A_{1}^{\sharp}A_{1}A_{1}^{m}(A_{1}^{k+1}+A_{1}^{k}A_{2}
+A_{1}^{k-1}A_{2}^{2}+\cdots+A_{1}^{2}A_{2}^{k-1})
\\
\nonumber
&
=
A_{1}^{k}+A_{1}^{k-1}A_{2}+A_{1}^{k-2}A_{2}^{2}
+\cdots+A_{1}A_{2}^{k-1}+A_{2}^{k}
\\
\label{WWL-MCE-001}
&
=
(A_{1}+A_{2})^{k}=A^{k},
\\
\nonumber
(A\mathfrak{X})^{\sim}
&=
((A_{1}+A_{2})A_{1}^{\sharp}A_{1}A_{1}^{m})^{\sim}
=
(A_{1}A_{1}^{\sharp}A_{1}A_{1}^{m}+A_{2}A_{1}^{\sharp}A_{1}A_{1}^{m})^{\sim}
\\
\nonumber
&
=
(A_{1}A_{1}^{\sharp}A_{1}A_{1}^{m}+A_{2}A_{1}A_{1}^{\sharp}A_{1}^{m})^{\sim}
=
(A_{1}A_{1}^{m})^{\sim}=A_{1}A_{1}^{m}
\\
\label{WWL-MCE-002}
&
=
A_{1}A_{1}^{\sharp}A_{1}A_{1}^{m}=(A_{1}
+A_{2})A_{1}^{\sharp}A_{1}A_{1}^{m}
=
A\mathfrak{X}
\end{align}
and
\begin{align}
\nonumber
\mathfrak{X}A\mathfrak{X}
&=
A_{1}^{\sharp}A_{1}A_{1}^{m}(A_{1}
+A_{2})A_{1}^{\sharp}A_{1}A_{1}^{m}
=
(A_{1}^{\sharp}A_{1}A_{1}^{m}A_{1}
+A_{1}^{\sharp}A_{1}A_{1}^{m}A_{2})A_{1}^{\sharp}A_{1}A_{1}^{m}
\\
\label{WWL-MCE-003}
&
=
A_{1}^{\sharp}A_{1}A_{1}^{\sharp}A_{1}A_{1}^{m}
+A_{1}^{\sharp}A_{1}A_{1}^{m}A_{2}A_{1}A_{1}^{\sharp}A_{1}^{m}
=
\mathfrak{X}.
\end{align}
By (\ref{WWL-MCE-5}),
(\ref{WWL-MCE-666})
and (\ref{WWL-MCE-6}),
we have
\begin{align}
\label{WWL-MCE-16}
A_{1}^{\sharp}=
U
\left[
\begin{matrix}
 T^{-1} &  T^{-2}S\\
   0     &  0
\end{matrix}
\right]
U^{*}.
\end{align}
It is easy to check that
\begin{align}\label{WWL-MCE-004}
\mathcal{R}\left(\mathfrak{X}\right)
\subseteq \mathcal{R}\left(A^{k}\right).
\end{align}
Therefore,
by applying
(\ref{WWL-MCE-001}),
(\ref{WWL-MCE-002}),
(\ref{WWL-MCE-003}),
(\ref{WWL-MCE-004})
and
Definition \ref{WWL-MCE-df-3.1},
we get (\ref{WWL-MCE-4-1}).
\end{proof}

\section{The ${\mathfrak{m}}$-core-EP Decomposition}

In this section,
we introduce a decomposition (called ${\mathfrak{m}}$-core-EP decomposition)
in Minkowski space,
prove that the decomposition is unique,
derive several characterizations of it,
and
apply it to study the ${\mathfrak{m}}$-core-EP inverse.

Let $A \in\mathbb{C}_{n,n}$
with ${\mbox{\rm Ind}}(A)=k$
be as in (\ref{WWL-MCE-5}),
 and
 ${\mbox{\rm rk}}\left(A^{k}\right)
 ={\mbox{\rm rk}}\left(\left(A^{k}\right)^{\sim}A^{k}\right)$.
Write
 \begin{align}
\label{WWL-MCE-224}
\widehat{A_{1}}=
U
\left[
\begin{matrix}
 T&  S+G_{1}^{-1}G_{2}N\\
   0     &  0
\end{matrix}
\right]
U^{*}
\ \mbox{ and } \
\widehat{A_{2}}=
U
\left[
\begin{matrix}
 0 &  -G_{1}^{-1}G_{2}N\\
   0     &  N
\end{matrix}
\right]
U^{*},
\end{align}
in which $U$, $T$, $N$ and $S$ are as given in (\ref{WWL-MCE-666}),
and $G_1$ and $G_2$ are as given in (\ref{WWL-MCE-4}).
By applying (\ref{WWL-MCE-2}) and (\ref{WWL-MCE-224}), we obtain
\begin{align}\nonumber
\widehat{A_{1}}^{\sim}\widehat{A_{1}}
&=
G\widehat{A_{1}}^{*}G\widehat{A_{1}}
\\
\nonumber
&
=
GU
\left[
\begin{matrix}
   T^{*}  &      0  \\
       S^{*}+N^{*}G_{2}^{*}(G_{1}^{-1})^{*}   &       0
\end{matrix}
\right]
U^{*}GU
\left[
\begin{matrix}
   T &       S+G_{1}^{-1}G_{2}N \\
       0     &       0
\end{matrix}
\right]
U^{*}
\\
\nonumber
&
=
GU
\left[
\begin{matrix}
   T^{*}  &      0  \\
       S^{*}+N^{*}G_{2}^{*}(G_{1}^{-1})^{*}   &       0
\end{matrix}
\right]
\left[
\begin{matrix}
   G_{1} &       G_{2} \\
       G_{3}     &       G_{4}
\end{matrix}
\right]
\left[
\begin{matrix}
   T &       S+G_{1}^{-1}G_{2}N \\
       0     &       0
\end{matrix}
\right]
U^{*}
\\
\nonumber
&
=
GU
\left[
\begin{matrix}
   T^{*}G_{1}T  &      T^{*}G_{1}S+T^{*}G_{2}N  \\
       S^{*}G_{1}T+N^{*}G_{2}^{*}T   &       S^{*}G_{1}S+S^{*}G_{2}N+N^{*}G_{2}^{*}S+N^{*}G_{2}^{*}G_{1}^{-1}G_{2}N
\end{matrix}
\right]
U^{*}.
\end{align}
Then, ${\mbox{\rm rk}}\left(\widehat{A_{1}}\right)
=
 {\mbox{\rm rk}}\left(T^{*}G_{1}T \right)
\leq
{\mbox{\rm rk}}\left(\widehat{A_{1}}^{\sim}\widehat{A_{1}}\right)
\leq
{\mbox{\rm rk}}\left(\widehat{A_{1}}\right)$,
that is,
\begin{align}
\label{WWL-MCE-123}
{\mbox{\rm rk}}\left(\widehat{A_{1}}^{\sim}\widehat{A_{1}}\right)
=
{\mbox{\rm rk}}\left(\widehat{A_{1}}\right).
\end{align}
Applying $N^k=0$,
we have
\begin{align}
\label{WWL-MCE-124}
\widehat{A_{2}}^{k}
&=
U
\left[
\begin{matrix}
   0 &     -G_{1}^{-1}G_{2}N    \\
       0 &       N
\end{matrix}
\right]^{k}
U^\ast
=
U
\left[
\begin{matrix}
  0  & -G_{1}^{-1}G_{2}N^{k}   \\
       0     &       N^{k}
\end{matrix}
\right]
U^\ast
=
0,
\\
\nonumber
\widehat{A_{1}}^{\sim}\widehat{A_{2}}
&=
G\widehat{A_{1}}^{*}G\widehat{A_{2}}
\\
\nonumber
&
=
GU
\left[
\begin{matrix}
   T^{*} &     0    \\
       S^{*}+N^{*}G_{2}^{*}(G_{1}^{-1})^{*} &       0
\end{matrix}
\right]
\left[
\begin{matrix}
   G_{1} &     G_{2}    \\
      G_{3}  &    G_{4}
\end{matrix}
\right]
\left[
\begin{matrix}
   0 &     -G_{1}^{-1}G_{2}N    \\
       0 &       N
\end{matrix}
\right]
U^\ast
\\
\label{WWL-MCE-125}
&
=
GU
\left[
\begin{matrix}
   0 &     0    \\
       0 &       0
\end{matrix}
\right]
U^\ast=0
\end{align}
and
\begin{align}
\label{WWL-MCE-126}
\widehat{A_{2}}\widehat{A_{1}}
&=
U
\left[
\begin{matrix}
   0 &     -G_{1}^{-1}G_{2}N    \\
       0 &       N
\end{matrix}
\right]
\left[
\begin{matrix}
   T &     S+G_{1}^{-1}G_{2}N    \\
       0 &       0
\end{matrix}
\right]
U^{*}=0.
\end{align}

By applying (\ref{WWL-MCE-123}),
(\ref{WWL-MCE-124}),
(\ref{WWL-MCE-125})
and (\ref{WWL-MCE-126}),
we get the following Theorem \ref{WWL-MCE-df4.1}.

\begin{theorem}
\label{WWL-MCE-df4.1}
{ \rm(The ${\mathfrak{m}}$-core-EP Decomposition)}.
Let $A \in\mathbb{C}_{n,n}$
with ${\mbox{\rm Ind}}(A)=k$
 and
 ${\mbox{\rm rk}}\left(A^{k}\right)
 ={\mbox{\rm rk}}\left(\left(A^{k}\right)^{\sim}A^{k}\right)
 =r$.
 Then $A$ can be written as
 the sum of matrices $\widehat{A_{1}}$
 and $\widehat{A_{2}}$,
 i.e.,
 $A=\widehat{A_{1}}+\widehat{A_{2}}$,
 where
\begin{enumerate}
  \item[{\footnotesize\rm(i)}]
$\widehat{A_{1}}\in\mathbb{C}_{n}^{{\tiny{\mbox{\rm CM}}}}$
 with
 ${\mbox{\rm rk}}\left(\widehat{A_{1}} \right)
 ={\mbox{\rm rk}}\left(\widehat{A_{1}}^{\sim}\widehat{A_{1}} \right)$;

\item[{\footnotesize\rm(ii)}]
$\widehat{A_{2}}^{k}=0$;

\item[{\footnotesize\rm(iii)}]
$\widehat{A_{1}}^{\sim}\widehat{A_{2}}=\widehat{A_{2}}\widehat{A_{1}}=0$.
\end{enumerate}
Furthermore,
$\widehat{A_1}$ and $\widehat{A_2}$ have the form (\ref{WWL-MCE-224}).
Here one or both of $\widehat{A_{1}}$ and $\widehat{A_{2}}$ can be null.
\end{theorem}

\begin{theorem}
\label{WWL-MCE-th4.2}
Let $A \in\mathbb{C}_{n,n}$ with ${\mbox{\rm Ind}}(A)=k$,
and
let the ${\mathfrak{m}}$-core-EP decomposition of $A$
be as in Theorem \ref{WWL-MCE-df4.1}.
Then
\begin{align}
\label{WWL-MCE-th4.2.1-2}
A^{\tiny\textcircled{E}}
=
\widehat{A_{1}}^{\tiny\textcircled{m}}.
\end{align}
\end{theorem}

\begin{proof}
Let
$\widehat{A_{1}}$
and
$\widehat{A_{2}}$
be as in Theorem \ref{WWL-MCE-df4.1}.
Applying (\ref{WWL-MCE-2}),
we have
\begin{align}
\label{WWL-MCE-22}
\widehat{A_{2}}\widehat{A_{1}}^{\tiny\textcircled{m}}
=
\widehat{A_{2}}\widehat{A_{1}}\left(\widehat{A_{1}}^{\tiny\textcircled{m}}\right)^{2}
=0 .
\end{align}
Therefore
\begin{align}
\widehat{A_{1}}^{\tiny\textcircled{m}}A\widehat{A_{1}}^{\tiny\textcircled{m}}
&=
\widehat{A_{1}}^{\tiny\textcircled{m}}\widehat{A_{1}}
\widehat{A_{1}}^{\tiny\textcircled{m}}+\widehat{A_{1}}^{\tiny\textcircled{m}}
\widehat{A_{ 2}}\widehat{A_{1}}^{\tiny\textcircled{m}}
\label{WWL-MCE-th4.2.1}
=
\widehat{A_{1}}^{\tiny\textcircled{m}}
\end{align}
and
\begin{align}
\label{WWL-MCE-th4.2.1-1}
\left(A\widehat{A_{1}}^{\tiny\textcircled{m}}\right)^{\sim}
=
\left(\widehat{A_{1}}\widehat{A_{1}}^{\tiny\textcircled{m}}\right)^{\sim}
=
A\widehat{A_{1}}^{\tiny\textcircled{m}}.
\end{align}
{{ From}}
${\mbox{\rm Ind}}(A)=k$,
$\widehat{A_{2}}^{k}=0$
and
$\widehat{A_{1}}^{\tiny\textcircled{m}}\widehat{A_{1}}
=
\widehat{A_{1}}^{\#}\widehat{A_{1}}$,
we have
\begin{align}\label{WWL-MCE-th4.2.2}
\nonumber
\widehat{A_{1}}^{\tiny\textcircled{m}}A^{k+1}
&
=
\widehat{A_{1}}^{\tiny\textcircled{m}}
\left(\widehat{A_{1}}^{k+1}
+
\widehat{A_{1}}^{k}\widehat{A_{2}}+\cdots+\widehat{A_{1}}\widehat{A_{2}}^{k}
+\widehat{A_{2}}^{k+1}\right)
\\
\nonumber
&
=
\widehat{A_{1}}^{\sharp}\widehat{A_{1}}\widehat{A_{1}}
\left(\widehat{A_{1}}^{k-1}
+\widehat{A_{1}}^{k-2}\widehat{A_{2}}
+\cdots+\widehat{A_{2}}^{k-1}\right)
\\
\nonumber
&
=
\widehat{A_{1}}^{k}+\widehat{A_{1}}^{k-1}\widehat{A_{2}}
+\cdots+\widehat{A_{1}}\widehat{A_{2}}^{k-1}
+\widehat{A_{2}}^{k}
\\
&
=
\left(\widehat{A_{1}}+\widehat{A_{2}}\right)^{k}=A^{k}.
\end{align}

Since
$\widehat{A_{1}}\left(\widehat{A_{1}}^{\tiny\textcircled{m}}\right)^{2}
=
\widehat{A_{1}}^{\tiny\textcircled{m}}$,
then
$
A\left(\widehat{A_{1}}^{\tiny\textcircled{m}}\right)^{2}
=
\left(\widehat{A_{1}}+\widehat{A_{2}}\right)
\left(\widehat{A_{1}}^{\tiny\textcircled{m}}\right)^{2}
=\widehat{A_{1}}^{\tiny\textcircled{m}}$, and
\begin{align}\nonumber
A^{k}\left(\widehat{A_{1}}^{\tiny\textcircled{m}}\right)^{k+1}
&
=
A^{k-1}A\left(\widehat{A_{1}}^{\tiny\textcircled{m}}\right)^{2}
\left(\widehat{A_{1}}^{\tiny\textcircled{m}}\right)^{k-1}
=
A^{k-1}\widehat{A_{1}}^{\tiny\textcircled{m}}
\left(\widehat{A_{1}}^{\tiny\textcircled{m}}\right)^{k-1}
\\
\nonumber
&
=
A^{k-2}A\left(\widehat{A_{1}}^{\tiny\textcircled{m}}\right)^{2}
\left(\widehat{A_{1}}^{\tiny\textcircled{m}}\right)^{k-2}
=
A^{k-2}\widehat{A_{1}}^{\tiny\textcircled{m}}
\left(\widehat{A_{1}}^{\tiny\textcircled{m}}\right)^{k-2}
=\cdots
\\
\nonumber
&
=
AA\left(\widehat{A_{1}}^{\tiny\textcircled{m}}\right)^{2}
\widehat{A_{1}}^{\tiny\textcircled{m}}
=
A\widehat{A_{1}}^{\tiny\textcircled{m}}\widehat{A_{1}}^{\tiny\textcircled{m}}
=\widehat{A_{1}}^{\tiny\textcircled{m}}.
\end{align}
Therefore,
\begin{align}
\label{WWL-MCE-th4.2.3}
\mathcal{R}\left(\widehat{A_{1}}^{\tiny\textcircled{m}}\right)
\subseteq
\mathcal{R}\left(A^{k}\right).
\end{align}

Therefore, applying
(\ref{WWL-MCE-th4.2.1}),
(\ref{WWL-MCE-th4.2.1-1}),
(\ref{WWL-MCE-th4.2.2})
and
(\ref{WWL-MCE-th4.2.3}),
we get (\ref{WWL-MCE-th4.2.1-2}).
\end{proof}

\begin{remark}
Let
$A\in\mathbb{C}_{n,n}$ with ${\mbox{\rm Ind}}(A)=k$
 and
 ${\mbox{\rm rk}}\left(A^{k}\right)
 ={\mbox{\rm rk}}\left(\left(A^{k}\right)^{\sim}A^{k}\right)$.
And let $A_1$ and $\widehat{A_{1}}$ be as given in
 Theorem \ref{WWL-Theorem-MCE-2}
 and
 Theorem \ref{WWL-MCE-df4.1}, respectively.
 It  is interesting to see that
 $ A^{\tiny\textcircled{E}}
=
\widehat{A_{1}}^{\tiny\textcircled{m}}
=
 A_{1}^{\tiny\textcircled{m}}$.
 \end{remark}

It can be observed from {{ Example \ref{WWL-MCE-example-3.1}
where $G_{1}=0$}}
is singular that,
after applying Lemma \ref{WWL-MCE-734}
and Theorem \ref{WWL-MCE-TH3.2},
a matrix has a core-EP decomposition,
but it not necessary has {\rm${\mathfrak{m}}$-core-EP} inverse
or {\rm${\mathfrak{m}}$-core-EP decomposition}.
The matrix has a {\rm${\mathfrak{m}}$-core-EP} decomposition,
if and only if the {\rm${\mathfrak{m}}$-core-EP} inverse exists.
Therefore,
{\rm${\mathfrak{m}}$-core-EP decomposition}
is different from core-EP decomposition.
\begin{example}
Let
$A
=
 \left[{\begin{matrix}
 \frac{16+4\sqrt5}{15} &  \frac{2+8\sqrt5}{15} & \frac{10-8\sqrt5}{15}  \\
   \frac{-8+3\sqrt5}{15} & \frac{-1+6\sqrt5}{15} & \frac{-5-6\sqrt5}{15} \\
   \frac{\sqrt5}{3} &  \frac{2\sqrt5}{3} & -\frac{2\sqrt5}{3}
\end{matrix} }\right]$ with ${\rm Ind}(A)=2$
 and  ${\mbox{\rm rk}}(A^2)=1$.
There exists a unitary matrix
$$U
=
\left[
\begin{matrix}
 \frac{2}{\sqrt5} &  \frac{2}{3\sqrt5} & \frac{1}{3}  \\
   -\frac{1}{\sqrt5} & \frac{4}{3\sqrt5} & \frac{2}{3} \\
   0 &  \frac{5}{3\sqrt5} & -\frac{2}{3}
\end{matrix}
\right],$$
such that
\begin{align*}
A
=
U
\left[
\begin{matrix}
 1 &  1 & 1  \\
   0 & 0 &3 \\
   0 &  0 & 0
\end{matrix}
\right]
U^{*}.
\end{align*}
By applying (\ref{WWL-MCE-4}), we have
\begin{align*}
U^{*}GU
=
\left[{\begin{matrix}
 G_{1}
    & G_{2}   \\
 G_{3}   & G_{4}   \\
\end{matrix} }\right]
=
\left[
\begin{matrix}
 \frac{3}{5} &  \frac{8}{15} & \frac{4}{3\sqrt5}  \\
   \frac{8}{15} & -\frac{37}{45} & \frac{4}{9\sqrt5} \\
   \frac{4}{3\sqrt5} &  \frac{4}{9\sqrt5} & -\frac{7}{9}
\end{matrix}
\right],
\end{align*}
where $G_1\in\mathbb{C}_{1,1}$.
Since $\frac{3}{5}$ is nonsingular,
by applying Lemma \ref{WWL-MCE-734}
and Theorem \ref{WWL-MCE-TH3.2},
we observe that $A$ is {${\mathfrak{m}}$}-core-EP invertible.
Then
{\small
\begin{align*}
A^{\tiny\textcircled{E}}
&
=
\left[
\begin{matrix}
 \frac{2}{\sqrt5} &  \frac{2}{3\sqrt5} & \frac{1}{3}  \\
   -\frac{1}{\sqrt5} & \frac{4}{3\sqrt5} & \frac{2}{3} \\
   0 &  \frac{5}{3\sqrt5} & -\frac{2}{3}
\end{matrix}
\right]
\left[
\begin{matrix}
 \frac{5}{3} &  0 & 0  \\
   0 & 0 &0 \\
   0 &  0 & 0
\end{matrix}
\right]
\left[
\begin{matrix}
 \frac{2}{\sqrt5} &  -\frac{1}{\sqrt5} & 0  \\
   \frac{2}{3\sqrt5} & \frac{4}{3\sqrt5} & \frac{5}{3\sqrt5} \\
   \frac{1}{3} &  \frac{2}{3} & -\frac{2}{3}
\end{matrix}
\right]
\left[
\begin{matrix}
 1 &  0 & 0  \\
   0 & -1 &0 \\
   0 &  0 & -1
\end{matrix}
\right]
=
\left[
\begin{matrix}
 \frac{4}{3} &  \frac{2}{3} &0  \\
   -\frac{2}{3} & -\frac{1}{3} & 0 \\
   0 &  0 & 0
\end{matrix}
\right],
\\
A_{1}
&
=
\left[
\begin{matrix}
 \frac{2}{\sqrt5} &  \frac{2}{3\sqrt5} & \frac{1}{3}  \\
   -\frac{1}{\sqrt5} & \frac{4}{3\sqrt5} & \frac{2}{3} \\
   0 &  \frac{5}{3\sqrt5} & -\frac{2}{3}
\end{matrix}
\right]
\left[
\begin{matrix}
 1 &  1 & 1  \\
   0 & 0 &0 \\
   0 &  0 & 0
\end{matrix}
\right]
\left[
\begin{matrix}
 \frac{2}{\sqrt5} &  -\frac{1}{\sqrt5} & 0  \\
   \frac{2}{3\sqrt5} & \frac{4}{3\sqrt5} & \frac{5}{3\sqrt5} \\
   \frac{1}{3} &  \frac{2}{3} & -\frac{2}{3}
\end{matrix}
\right]
=
\left[
\begin{matrix}
 \frac{16+2\sqrt5}{15} &  \frac{2+4\sqrt5}{15} &\frac{10-4\sqrt5}{15}  \\
   \frac{-8-\sqrt5}{15} & \frac{-1-2\sqrt5}{15} & \frac{-5+2\sqrt5}{15} \\
   0 &  0 & 0
\end{matrix}
\right],
\\
A_{1}^{\tiny\textcircled{m}}
&
=
\left[
\begin{matrix}
 \frac{2}{\sqrt5} &  \frac{2}{3\sqrt5} & \frac{1}{3}  \\
   -\frac{1}{\sqrt5} & \frac{4}{3\sqrt5} & \frac{2}{3} \\
   0 &  \frac{5}{3\sqrt5} & -\frac{2}{3}
\end{matrix}
\right]
\left[
\begin{matrix}
 \frac{5}{3} &  0 & 0  \\
   0 & 0 &0 \\
   0 &  0 & 0
\end{matrix}
\right]
\left[
\begin{matrix}
 \frac{2}{\sqrt5} &  -\frac{1}{\sqrt5} & 0  \\
   \frac{2}{3\sqrt5} & \frac{4}{3\sqrt5} & \frac{5}{3\sqrt5} \\
   \frac{1}{3} &  \frac{2}{3} & -\frac{2}{3}
\end{matrix}
\right]
\left[
\begin{matrix}
 1 &  0 & 0  \\
   0 & -1 &0 \\
   0 &  0 & -1
\end{matrix}
\right]
=
\left[
\begin{matrix}
 \frac{4}{3} &  \frac{2}{3} &0  \\
   -\frac{2}{3} & -\frac{1}{3} & 0 \\
   0 &  0 & 0
\end{matrix}
\right],\\
\widehat{A_{1}}
&
=
\left[
\begin{matrix}
 \frac{2}{\sqrt5} &  \frac{2}{3\sqrt5} & \frac{1}{3}  \\
   -\frac{1}{\sqrt5} & \frac{4}{3\sqrt5} & \frac{2}{3} \\
   0 &  \frac{5}{3\sqrt5} & -\frac{2}{3}
\end{matrix}
\right]
\left[
\begin{matrix}
 1 &  1 & \frac{11}{3}  \\
   0 & 0 &0 \\
   0 &  0 & 0
\end{matrix}
\right]
\left[
\begin{matrix}
 \frac{2}{\sqrt5} &  -\frac{1}{\sqrt5} & 0  \\
   \frac{2}{3\sqrt5} & \frac{4}{3\sqrt5} & \frac{5}{3\sqrt5} \\
   \frac{1}{3} &  \frac{2}{3} & -\frac{2}{3}
\end{matrix}
\right]
=
\left[
\begin{matrix}
 \frac{48+22\sqrt5}{45} &  \frac{6+44\sqrt5}{45} &\frac{30-44\sqrt5}{45}  \\
   \frac{-24-11\sqrt5}{45} & \frac{-3-22\sqrt5}{45} & \frac{-15+22\sqrt5}{45} \\
   0 &  0 & 0
\end{matrix}
\right],
\\
\widehat{A_{1}}^{\tiny\textcircled{m}}
&
=
\left[
\begin{matrix}
 \frac{2}{\sqrt5} &  \frac{2}{3\sqrt5} & \frac{1}{3}  \\
   -\frac{1}{\sqrt5} & \frac{4}{3\sqrt5} & \frac{2}{3} \\
   0 &  \frac{5}{3\sqrt5} & -\frac{2}{3}
\end{matrix}
\right]
\left[
\begin{matrix}
 \frac{5}{3} &  0 & 0  \\
   0 & 0 &0 \\
   0 &  0 & 0
\end{matrix}
\right]
\left[
\begin{matrix}
 \frac{2}{\sqrt5} &  -\frac{1}{\sqrt5} & 0  \\
   \frac{2}{3\sqrt5} & \frac{4}{3\sqrt5} & \frac{5}{3\sqrt5} \\
   \frac{1}{3} &  \frac{2}{3} & -\frac{2}{3}
\end{matrix}
\right]
\left[
\begin{matrix}
 1 &  0 & 0  \\
   0 & -1 &0 \\
   0 &  0 & -1
\end{matrix}
\right]
=
\left[
\begin{matrix}
 \frac{4}{3} &  \frac{2}{3} &0  \\
   -\frac{2}{3} & -\frac{1}{3} & 0 \\
   0 &  0 & 0
\end{matrix}
\right].
\end{align*}}
It can be observed from above that
$\widehat{A_{1}}\neq A_{1}$.
However,
$ A^{\tiny\textcircled{E}}
=
\widehat{A_{1}}^{\tiny\textcircled{m}}
=
A_{1}^{\tiny\textcircled{m}}$.
\end{example}

\begin{theorem}
Let $A \in\mathbb{C}_{n,n}$
with ${\mbox{\rm Ind}}(A)=k$
 and
 ${\mbox{\rm rk}}\left(A^{k}\right)
 ={\mbox{\rm rk}}\left(\left(A^{k}\right)^{\sim}A^{k}\right)
 =r$.
The ${\mathfrak{m}}$-core-EP decomposition of $A$ is unique.
\end{theorem}

\begin{proof}
 Suppose that
 $A=\widehat{A_{1}}+\widehat{A_{2}}$
 is the ${\mathfrak{m}}$-core-EP decomposition of $A$.
 Let $A=\widehat{B_{1}}+\widehat{B_{2}}$
  be another ${\mathfrak{m}}$-core-EP decomposition of $A$.
By Theorem \ref{WWL-MCE-th4.2},
we know that
\begin{align}\label{WWL-MCE-24}
\widehat{A_{1}}^{\tiny\textcircled{m}}
=
\widehat{B_{1}}^{\tiny\textcircled{m}}
=A^{\tiny\textcircled{E}}.
\end{align}
Premultiplying both sides of (\ref{WWL-MCE-24}) with $A$,
then
 \begin{align}
\widehat{A_{1}}\widehat{A_{1}}^{\tiny\textcircled{m}}
+\widehat{A_{2}}\widehat{A_{1}}^{\tiny\textcircled{m}}
=\widehat{B_{1}}\widehat{B_{1}}^{\tiny\textcircled{m}}
+\widehat{B_{2}}\widehat{B_{1}}^{\tiny\textcircled{m}}.
\end{align}
Since
$\widehat{A_{2}}\widehat{A_{1}}^{\tiny\textcircled{m}}=0$
and
$\widehat{B_{2}}\widehat{B_{1}}^{\tiny\textcircled{m}}=0$,
we get
\begin{align}
\label{WWL-MCE-26}
\widehat{A_{1}}\widehat{A_{1}}^{\tiny\textcircled{m}}
=\widehat{B_{1}}\widehat{B_{1}}^{\tiny\textcircled{m}}.
\end{align}
Postmultiplying both sides of (\ref{WWL-MCE-26}) with $A$,
then
\begin{align}
\widehat{A_{1}}\widehat{A_{1}}^{\tiny\textcircled{m}}\widehat{A_{1}}
+
\widehat{A_{1}}\widehat{A_{1}}^{\tiny\textcircled{m}}\widehat{A_{2}}
=
\widehat{B_{1}}\widehat{B_{1}}^{\tiny\textcircled{m}}\widehat{B_{1}}
+
\widehat{B_{1}}\widehat{B_{1}}^{\tiny\textcircled{m}}\widehat{B_{2}}.
\end{align}
Because
  the m-core inverse $\widehat{A_{1}}^{\tiny\textcircled{m}}$
satisfies
$\widehat{A_{1}}^{\tiny\textcircled{m}}A_1\widehat{A_{1}}^{\tiny\textcircled{m}}
=
\widehat{A_{1}}^{\tiny\textcircled{m}}$,
applying
$A_1\widehat{A_{1}}^{\tiny\textcircled{m}}
=
\left(A_1\widehat{A_{1}}^{\tiny\textcircled{m}}\right)^{\sim}$, we get
$\widehat{A_{1}}^{\tiny\textcircled{m}}\widehat{A_{2}}=0$.
In the same way,
we have
$\widehat{B_{1}}^{\tiny\textcircled{m}}\widehat{B_{2}}=0$.
It follows that
$\widehat{A_{1}}=\widehat{B_{1}}$,
that is,
the ${\mathfrak{m}}$-core-EP decomposition of a given matrix is unique.
\end{proof}

\bigskip

\begin{theorem}
Let $A \in\mathbb{C}_{n,n}$
with
${\mbox{\rm Ind}}(A)=k$ and
${\mbox{\rm rk}}\left(A^{k}\right)
=
{\mbox{\rm rk}}\left(\left(A^{k}\right)^{\sim}A^{k}\right)=r$,
and let the ${\mathfrak{m}}$-core-EP decomposition of $A$
be as in Theorem \ref{WWL-MCE-df4.1}.
Then
 \begin{align}
\label{in-eq-5}
\widehat{A_{1}}
=
A^{k}\left(A^{k}\right)^{\tiny\textcircled{m}}A
\ \ and \ \
\widehat{A_{2}}
=
A-A^{k}\left(A^{k}\right)^{\tiny\textcircled{m}}A.
\end{align}
\end{theorem}

\begin{proof}
By applying
(\ref{WWL-MCE-5}), (\ref{WWL-MCE-6}) and (\ref{WWL-MCE-4}),
we have
\begin{align}
\nonumber
A^{k}(A^{k})^{\tiny\textcircled{m}}A
&
=
U
\left[
\begin{matrix}
   T^{k} &     \widehat{T}    \\
       0 &       0
\end{matrix}
\right]
U^{*}U
\left[
\begin{matrix}
   T^{-k}G_{1}^{-1} &     0    \\
       0 &       0
\end{matrix}
\right]
U^{*}GU
\left[
\begin{matrix}
   T &     S    \\
       0 &       N
\end{matrix}
\right]
U^{*}
\\
\nonumber
&
=
U
\left[
\begin{matrix}
   T^{k} &     \widehat{T}    \\
       0 &       0
\end{matrix}
\right]
\left[
\begin{matrix}
   T^{-k}G_{1}^{-1} &     0    \\
       0 &       0
\end{matrix}
\right]
\left[
\begin{matrix}
   G_{1} &     G_{2}    \\
       G_{3} &       G_{4}
\end{matrix}
\right]
\left[
\begin{matrix}
   T &     S    \\
       0 &       N
\end{matrix}
\right]
U^{*}
\\
\nonumber
&
=
U
\left[
\begin{matrix}
   T &     S+G_{1}^{-1}G_{2}N    \\
       0 &       0
\end{matrix}
\right]
U^\ast
=
\widehat{A_{1}},
\end{align}
where $G_{i}$($i=1,2,3,4$) are given as in (\ref{WWL-MCE-4}).
Since
$A=\widehat{A_{1}}+\widehat{A_{2}}$,
$\widehat{A_{2}}=A-\widehat{A_{1}}
=
A-A^{k}\left(A^{k}\right)^{\tiny\textcircled{m}}A$.
Therefore,
we get (\ref{in-eq-5}).
\end{proof}

\begin{theorem}
Let $A \in\mathbb{C}_{n,n}$
with ${\mbox{\rm Ind}}(A)=k$ and
${\mbox{\rm rk}}\left(A^{k}\right)
=
{\mbox{\rm rk}}\left(\left(A^{k}\right)^{\sim}A^{k}\right)$,
and let the ${\mathfrak{m}}$-core-EP decomposition of $A$
be as in Theorem \ref{WWL-MCE-df4.1}.
Then
 \begin{align}
\label{in-eq-5-20200117}
\widehat{A_{1}}=AA^{\tiny\textcircled{E}}A
 \ \ and \ \
\widehat{A_{2}}=A-AA^{\tiny\textcircled{E}}A.
\end{align}
\end{theorem}

\begin{proof}
By applying (\ref{WWL-MCE-5})
and (\ref{WWL-MCE-th3.2}),
we have
\begin{align}\nonumber
AA^{\tiny\textcircled{E}}A
&=
U
\left[
\begin{matrix}
   T &     S    \\
       0 &       N
\end{matrix}
\right]
U^{*}U
\left[
\begin{matrix}
   T^{-1}G_{1}^{-1} &     0    \\
       0 &       0
\end{matrix}
\right]
U^{*}GU
\left[
\begin{matrix}
   T &     S    \\
       0 &       N
\end{matrix}
\right]
U^{*}
\\
\nonumber
&
=
U
\left[
\begin{matrix}
   T &     S    \\
       0 &       N
\end{matrix}
\right]
\left[
\begin{matrix}
   T^{-1}G_{1}^{-1} &     0    \\
       0 &       0
\end{matrix}
\right]
\left[
\begin{matrix}
   G_{1} &     G_{2}    \\
       G_{3} &       G_{4}
\end{matrix}
\right]
\left[
\begin{matrix}
   T &     S    \\
       0 &       N
\end{matrix}
\right]
U^{*}
\\
\nonumber
&
=
U
\left[
\begin{matrix}
   T &     S+G_{1}^{-1}G_{2}N    \\
       0 &       0
\end{matrix}
\right]
U^\ast
=
\widehat{A_{1}},
\end{align}
where $G_{i}$($i=1,2,3,4$) are given as in (\ref{WWL-MCE-4}).
Since
$A=\widehat{A_{1}}+\widehat{A_{2}}$,
$\widehat{A_{2}}
=A-\widehat{A_{1}}
=A-AA^{\tiny\textcircled{E}}A$.
Therefore,
we get (\ref{in-eq-5-20200117}).
\end{proof}

\section{The {${\mathfrak{m}}$}-core-EP  order}
{{ In \cite{Wang2019laa299},}}
Wang, Li and Liu considered the {${\mathfrak{m}}$}-core partial order
by applying the   {${\mathfrak{m}}$}-core inverse in $\mathcal{M}$, which is characterized by
\begin{align}
\label{WWL-MCE-3022}
A \mathop \leq \limits ^ {\tiny\textcircled{m}} B
\Leftrightarrow A^{\tiny\textcircled{m}}A
=A^{\tiny\textcircled{m}}B
  \ \  and  \ \
AA^{\tiny\textcircled{m}}
=BA^{\tiny\textcircled{m}}.
\end{align}
Furthermore, it has following property.
{{ \begin{lemma}
[\cite{Wang2019laa299}]
Let $A, B\in\mathbb{C}^{\tiny{\mbox{\rm CM}}}_n $
with
${\mbox{\rm rk}}\left(A^{\sim}A\right)={\mbox{\rm rk}}\left(A\right)=r>0$
and
${\mbox{\rm rk}}\left(B^{\sim}B\right)={\mbox{\rm rk}}\left(B\right)=s\geq r$. If $A \mathop \leq \limits ^ {\tiny\textcircled{m}} B$,
then
\begin{align}
\label{WWL-MCE-2222}
A^{\tiny\textcircled{m}}BB^{\tiny\textcircled{m}}=A^{\tiny\textcircled{m}}~and~B^{\tiny\textcircled{m}}BA^{\tiny\textcircled{m}}=A^{\tiny\textcircled{m}}.
\end{align}
\end{lemma}}}
In this section,
we introduce a new  order (called
the {${\mathfrak{m}}$}-core-EP  order),
consider its properties
and
get some characterizations of it.
It is true that
the {${\mathfrak{m}}$}-core-EP  order
is a generalization of
the {${\mathfrak{m}}$}-core partial order,
but it is a pre-order not a partial order.
Let $A ,B\in\mathbb{C}_{n,n}$,
with ${\mbox{\rm rk}}(A^{k})
={\mbox{\rm rk}}((A^{k})^{\sim}A^{k})
=r$.
We define the {${\mathfrak{m}}$}-core-EP order,
which is characterized by
\begin{align}
\label{WWL-MCE-30}
A \mathop \leq \limits ^ {\tiny\textcircled{E}} B
\Leftrightarrow A^{\tiny\textcircled{E}}A
=A^{\tiny\textcircled{E}}B
  \ \  and  \ \
AA^{\tiny\textcircled{E}}
=BA^{\tiny\textcircled{E}}.
\end{align}

\begin{theorem}
Let $A, B\in\mathbb{C}_{n,n}$,
 ${\mbox{\rm Ind}}(A)=k$,
  ${\mbox{\rm Ind}}(B)=t$,
${\mbox{\rm rk}}\left(A^{k}\right)
={\mbox{\rm rk}}\left(\left(A^{k}\right)^{\sim}A^{k}\right)=r$
and
${\mbox{\rm rk}}\left(B^{t}\right)
={\mbox{\rm rk}}\left(\left(B^{t}\right)^{\sim}B^{t}\right)=s\geq r$.
If $A \mathop \leq \limits ^ {\tiny\textcircled{E}} B$,
then there exists a unitary matrix
$\widehat{U}$
such that
{\footnotesize \begin{align}
\label{WWL-MCE-5.6}
A
&
=
\widehat{U}
\left[
\begin{matrix}
       T   &  S_1  &  S_2   \\
       0  &       \widehat{N_{11}}& \widehat{N_{12}}\\
       0  &       \widehat{N_{13}} &\widehat{N_{14}}
\end{matrix}
\right]
\widehat{U}^\ast ,
\\
\label{WWL-MCE-5.7}
B
&
=
\widehat{U}
\left[
\begin{matrix}
       T   &
       S_1 +
       G_1^{-1}
       {\left[\begin{matrix}
       G_{21} &  G_{22}
       \end{matrix}\right]}
       {\left[\begin{matrix}
       \widehat{N_{11}}-\widetilde{T}
       \\
       \widehat{N_{13}}
       \end{matrix}\right]}
           & S_2
           + G_1^{-1}
                  {\left[\begin{matrix}
       G_{21} &  G_{22}
       \end{matrix}\right]}
       {\left[\begin{matrix}
       \widehat{N_{12}}-\widehat{S }
       \\
       \widehat{N_{14}}-\widehat{N}
       \end{matrix}\right]}
        \\
       0  &      \widetilde{T}& \widehat{S}\\
       0  &       0 &\widehat{N}
\end{matrix}
\right]
\widehat{U}^\ast ,
\end{align}}
where
$\left[
\begin{matrix}
     \widehat{N_{11}}  &  \widehat{N_{12}}   \\
      \widehat{N_{13}}&  \widehat{N_{14}}
\end{matrix}
\right]\in \mathbb{C}_{n-r,n-r}$
and $\widehat{N}\in \mathbb{C}_{n-s,n-s}$
are nilpotent,
and
$ {T}\in \mathbb{C}_{r,r}$,
$\widetilde{T}\in \mathbb{C}_{s-r,s-r}$,
$G_1\in \mathbb{C}_{r,r}$ and
$\widehat{G_1}=\left[
\begin{matrix}
       G_1   &  G_{21}   \\
       G_{31}&  G_{41}   \\
\end{matrix}
\right]\in \mathbb{C}_{s,s}$
are invertible,
and
\begin{align}
\label{MC-1-3}
\widehat{U}^\ast G\widehat{U}
=
\left[
\begin{matrix}
       G_1   &  G_{21}  &   G_{22}  \\
       G_{31}&  G_{41}  &   G_{42}  \\
       G_{32}&  G_{43}  &   G_{44}  \\
\end{matrix}
\right].
\end{align}
\end{theorem}

\begin{proof}
Let $A$ be of the form (\ref{WWL-MCE-5}).
Applying (\ref{WWL-MCE-30})
and
Theorem \ref{WWL-MCE-TH3.2},
we observe that
\begin{align}
\label{MC-20}
A A^{{\tiny\textcircled{E}}}
=
U
\left[
\begin{matrix}
      G_1^{-1}  &       0 \\
       0        &       0
\end{matrix}
\right]
U^\ast G
\
\mbox{ \rm and }
\
A^{{\tiny\textcircled{E}}} A
=
U
\left[
\begin{matrix}
      I_r  &       T^{-1} S +T^{-1}G_{1}^{-1}G_{2}N\\
       0   &          0
\end{matrix}
\right]
U^\ast .
\end{align}

Suppose that
$U^\ast B U\in \mathbb{C}_{n,n}$ is
partitioned as
\begin{align}
 U^\ast B U
=
\left[
\begin{matrix}
       Q  &       Z \\
       M  &       P
\end{matrix}
\right],
\end{align}
where
$ {Q}\in \mathbb{C}_{r,r}$.
Applying (\ref{WWL-MCE-4}), we give
\begin{equation}
\label{MC-21}
\begin{aligned}
B A^{{\tiny\textcircled{E}}}
&
=
U
\left[
\begin{matrix}
      Q(G_1  T )^{-1}  &    0 \\
      M(G_1 T )^{-1}  &   0
\end{matrix}
\right]
U^\ast G ,
\\
A^{{\tiny\textcircled{E}}} B
&
=
U
\left[
\begin{matrix}
       T ^{-1}Q + (G_1 T )^{-1} G_2 M
       &      T ^{-1} Z + (G_1 T )^{-1}G_2 P \\
                       0
       &                    0
\end{matrix}
\right].
\end{aligned}
\end{equation}
Since
$A A^{{\tiny\textcircled{E}}}
= B A^{{\tiny\textcircled{E}}}$
and
$A^{{\tiny\textcircled{E}}} A
= A^{{\tiny\textcircled{E}}} B$,
we derive
$Q =T$,
 $M = 0$
 and
 $Z = S-G_{1}^{-1}G_{2}P+G_{1}^{-1}G_{2}N$.
 Therefore,
\begin{align}
\label{min-orde-eq-2}
B
&
=
U
\left[
\begin{matrix}
       T   &      S-G_{1}^{-1}G_{2}P+G_{1}^{-1}G_{2}N \\
       0  &       P
\end{matrix}
\right]
U^\ast.
\end{align}

Let
\begin{align}
\label{MC-1-1}
P=
U_1
\left[
\begin{matrix}
\widetilde{T}  &    \widehat{S}  \\
   0         &  \widehat{N}
\end{matrix}
\right]
U_1^\ast ,
\end{align}
where $\widetilde{T}$ is invertible,
$\widehat{N}$ is nilpotent,
and $U_{1}$ is unitary.
Denote
\begin{align}
\label{MC-1-2}
\widehat{U}=U\left[
\begin{matrix}
   I_r &    0  \\
   0   &    U_1
\end{matrix}
\right].
\end{align}

Applying
 (\ref{WWL-MCE-5}) and (\ref{WWL-MCE-666}),
we have
\begin{align}
\label{WWL-MCE-132}
A&
=
\widehat{U}
\left[
\begin{matrix}
{T}  &    {S} U_1 \\
   0         &  U_{1}^{*}NU_{1}
\end{matrix}
\right]
\widehat{U}^\ast
\ \ {\rm and } \ \
N
=
{U_{1}}
\left[
\begin{matrix}
\widehat{N_{11}}  &    \widehat{N_{12}} \\
   \widehat{N_{13}}         &  \widehat{N_{14}}
\end{matrix}
\right]
{U_{1}^\ast}.
\end{align}
Applying (\ref{WWL-MCE-132}) and (\ref{WWL-MCE-4}) ,
we have
\begin{align}
\label{MC-1-4}
\widehat{U}^\ast G\widehat{U}
&
=
\left[
\begin{matrix}
   I_r &    0  \\
   0   &    U_1^\ast
\end{matrix}
\right]U^\ast
GU\left[
\begin{matrix}
   I_r &    0  \\
   0   &    U_1
\end{matrix}
\right]
=
\left[
\begin{matrix}
  G_1   &  G_2 U_1 \\
    U_1^\ast G_3   &    U_1^\ast G_4 U_1
\end{matrix}
\right].
\end{align}
Furthermore, denote
\begin{align}
\label{MC-1-5}
\begin{aligned}
S U_1
&
=
\left[
\begin{matrix}
S_{1}  &   S_{2}
\end{matrix}
\right],
\
G_2 U_1
=
\left[
\begin{matrix}
G_{21}  &   G_{22}
\end{matrix}
\right],
\
\\
 U_1^\ast G_3
&
=
\left[
\begin{matrix}
G_{31}  \\   G_{32}
\end{matrix}
\right],
\ {\rm and } \
 U_1^\ast G_4 U_1
=
\left[
\begin{matrix}
 G_{41}  &   G_{42}  \\
 G_{43}  &   G_{44}  \\
\end{matrix}
\right],
\end{aligned}
\end{align}
where
$G_{41}\in \mathbb{C}_{s-r,s-r}$.
Substituting
(\ref{MC-1-5})
into
(\ref{MC-1-4}),
we obtain (\ref{MC-1-3}).

\bigskip

  By applying
Lemma \ref{WWL-MCE-734}
to
${\rm rk}\left(\left(A^{k}\right)^\sim A^{k}\right)
=
{\rm rk}\left(A^{k}\right)$
and
${\rm rk}\left(\left(B^{t}\right)^\sim B^{t}\right)
=
{\rm rk}\left(B^{t}\right)$,
we conclude that
$G_1\in \mathbb{C}_{r,r}$,
 $\left[
\begin{matrix}
       G_1   &  G_{21}   \\
       G_{31}&  G_{41}   \\
\end{matrix}
\right]\in \mathbb{C}_{s,s}$
and
$\widetilde{T}\in \mathbb{C}_{s-r,s-r}$ are invertible.

\bigskip

Substituting
(\ref{MC-1-5})
and
(\ref{MC-1-1}) into (\ref{min-orde-eq-2}), and $G_{1}$ is given as in (\ref{WWL-MCE-4}),
we obtain
\begin{align*}
B
&
=
U
\left[
\begin{matrix}
       T   &      S - G_1^{-1} G_2 {U_1
\left[
\begin{matrix}
\widetilde{T}  &    \widehat{S}  \\
   0         &  \widehat{N}
\end{matrix}
\right]
U_1^\ast }+G_{1}^{-1}G_{2}N \\
       0  &       U_{1}{
\left[
\begin{matrix}
\widetilde{T}  &    \widehat{S}  \\
   0         &  \widehat{N}
\end{matrix}
\right]
U_1^\ast}
\end{matrix}
\right]
U^{*}%
\\
&
=
\widehat{U}
\left[
\begin{matrix}
       T   &         \left[
\begin{matrix}
S_{1}  &   S_{2}
\end{matrix}
\right] - G_1^{-1}\left[
\begin{matrix}
G_{21}  &   G_{22}
\end{matrix}
\right]\left[
\begin{matrix}
\widetilde{T}  &    \widehat{S}  \\
   0         &  \widehat{N}
\end{matrix}
\right]+G_{1}^{-1}(G_{2}U_{1})(U_{1}^\ast NU_{1})\\
       0  &       {
\left[
\begin{matrix}
\widetilde{T}  &    \widehat{S}  \\
   0         &  \widehat{N}
\end{matrix}
\right]}
\end{matrix}
\right]
\widehat{U}^\ast,
\end{align*}
that is, (\ref{WWL-MCE-5.7}).
\end{proof}

\begin{theorem}
Let $A, B\in\mathbb{C}_{n,n}$,
 ${\mbox{\rm Ind}}(A)=k$,
  ${\mbox{\rm Ind}}(B)=t$,
${\mbox{\rm rk}}\left(A^{k}\right)
={\mbox{\rm rk}}\left(\left(A^{k}\right)^{\sim}A^{k}\right)=r$
and
${\mbox{\rm rk}}\left(B^{t}\right)
={\mbox{\rm rk}}\left(\left(B^{t}\right)^{\sim}B^{t}\right)=s\geq r$.
Then
\begin{align*}
A \mathop \leq \limits ^ {\tiny\textcircled{E}} B
\Leftrightarrow A^{k+1}
=
BA^{k}
~and~
A^{\sim}A^{k}
=
B^{\sim}A^{k}.
\end{align*}
\end{theorem}

\begin{proof}
 '$\Leftarrow$'
 {{ Let $A$ be of the form (\ref{WWL-MCE-5}).
 Applying
(\ref{WWL-MCE-th3.2}),}}
 we have
\begin{align}
\nonumber
(A^{k+1})^{\tiny\textcircled{m}}
=
U
\left[
\begin{matrix}
 (T^{k+1})^{-1}G_{1}^{-1} &  0\\
   0     &  0
\end{matrix}
\right]
U^{*}G, \ \
A^{\tiny\textcircled{E}}=
U
\left[
\begin{matrix}
 T^{-1}G_{1}^{-1} &  0\\
   0     &  0
\end{matrix}
\right]
U^{*}G
\end{align}
and
\begin{align}
\nonumber
A^{k}(A^{k+1})^{\tiny\textcircled{m}}
&=
U
\left[
\begin{matrix}
   T^{k}  &      \widehat{T}  \\
       0   &       0
\end{matrix}
\right]
U^{*}U
\left[
\begin{matrix}
   (T^{k+1})^{-1}G_{1}^{-1} &       0 \\
       0     &       0
\end{matrix}
\right]
U^{*}G
=
A^{\tiny\textcircled{E}}.
\end{align}

Since
$A^{k+1}(A^{k+1})^{\tiny\textcircled{m}}
=
AA^{k}(A^{k+1})^{\tiny\textcircled{m}}
=
BA^{k}(A^{k+1})^{\tiny\textcircled{m}}$,
 we get
$AA^{\tiny\textcircled{E}}
=BA^{\tiny\textcircled{E}}$.

Since
$A^{\sim}A^{k}=B^{\sim}A^{k}$
and
$A^{\sim}A^{k}(A^{k+1})^{\tiny\textcircled{m}}
=
B^{\sim}A^{k}(A^{k+1})^{\tiny\textcircled{m}}$,
we get
$A^{\sim}A^{\tiny\textcircled{E}}
=
B^{\sim}A^{\tiny\textcircled{E}}$. %\end{align}
Taking the Minkowski transpose of both sides,
we have
$(A^{\tiny\textcircled{E}})^{\sim}A
=
(A^{\tiny\textcircled{E}})^{\sim}B$.
Premultiplying both sides by
$A^{\tiny\textcircled{E}}A^{\sim}$,
then
$A^{\tiny\textcircled{E}}A^{\sim}(A^{\tiny\textcircled{E}})^{\sim}A
=
A^{\tiny\textcircled{E}}A^{\sim}(A^{\tiny\textcircled{E}})^{\sim}B$.
{{ From}}
\begin{align}
\nonumber
A^{\sim}
&
=
GA^{*}G
=
GU
\left[
\begin{matrix}
   T^{*}  &      0  \\
       S^{*}   &       N^{*}
\end{matrix}
\right]
U^{*}G,
\\
\nonumber
(A^{\tiny\textcircled{E}})^{\sim}
&=
GG^{*}U
\left[
\begin{matrix}
   (T^{-1}G_{1}^{-1})^{*}  &      0  \\
       0   &       0
\end{matrix}
\right]
U^{*}G
=
U
\left[
\begin{matrix}
   (T^{-1}G_{1}^{-1})^{*}  &      0  \\
       0   &       0
\end{matrix}
\right]
U^{*}G,
\end{align}
and
\begin{align}
\nonumber
A^{\tiny\textcircled{E}}
A^{\sim}(A^{\tiny\textcircled{E}})^{\sim}
&=
U
\left[
\begin{matrix}
   T^{-1}G_{1}^{-1}  &      0  \\
       0   &       0
\end{matrix}
\right]
U^{*}GGU
\left[
\begin{matrix}
   T^{*}  &      0  \\
       S^{*}   &       N^{*}
\end{matrix}
\right]
U^{*}GU
\left[
\begin{matrix}
   (T^{-1}G_{1}^{-1})^{*}  &      0  \\
       0   &       0
\end{matrix}
\right]
U^{*}G
\\
\nonumber
&
=
U
\left[
\begin{matrix}
   T^{-1}G_{1}^{-1}  &      0  \\
       0   &       0
\end{matrix}
\right]
\left[
\begin{matrix}
   T^{*}  &      0  \\
       S^{*}   &       N^{*}
\end{matrix}
\right]
\left[
\begin{matrix}
   G_{1} &      G_{2}  \\
       G_{3}   &       G_{4}
\end{matrix}
\right]
\left[
\begin{matrix}
   (T^{-1}G_{1}^{-1})^{*}  &      0  \\
       0   &       0
\end{matrix}
\right]
U^{*}G
\\
\nonumber
&
=
U
\left[
\begin{matrix}
   T^{-1}G_{1}^{-1}  &      0  \\
       0   &       0
\end{matrix}
\right]
U^{*}G
=
A^{\tiny\textcircled{E}},
\end{align}
we get
$A^{\tiny\textcircled{E}}A
=
A^{\tiny\textcircled{E}}B$.

\bigskip

'$\Rightarrow$'
{{ Let $A=A_1+A_2$ be the core-EP decomposition of $A$,}}
 and
$A_1$ and $A_2$ be as in  (\ref{WWL-MCE-666}).
Applying (\ref{WWL-MCE-30}), we have $B$ of the form (\ref{min-orde-eq-2}).
Then
\begin{align}\nonumber
BA^{k}
&=
U
\left[
\begin{matrix}
   T  &      S-G_{1}^{-1}G_{2}P+G_{1}^{-1}G_{2}N  \\
       0   &       P
\end{matrix}
\right]
U^{*}U
\left[
\begin{matrix}
   T^{k}  &      \widehat{T}  \\
       0   &       0
\end{matrix}
\right]
U^{*}
\\
\nonumber
&
=
U
\left[
\begin{matrix}
   T^{k+1}  &      T\widehat{T}  \\
       0   &       0
\end{matrix}
\right]
U^{*}
=
U
\left[
\begin{matrix}
   T^{k+1}  &      \overline{T}  \\
       0   &       0
\end{matrix}
\right]
U^{*}
=
A^{k+1},
\\
\nonumber
A^{\sim}A^{k}
&=
GU
\left[
\begin{matrix}
   T^{*}  &      0  \\
       S^{*}   &       N^{*}
\end{matrix}
\right]
U^{*}GU
\left[
\begin{matrix}
   T^{k}  &      \widehat{T}  \\
       0   &       0
\end{matrix}
\right]
U^{*}
=
GU
\left[
\begin{matrix}
   T^{*}  &      0  \\
       S^{*}   &       N^{*}
\end{matrix}
\right]
\left[
\begin{matrix}
   G_{1}  &      G_{2}  \\
       G_{3}   &       G_{4}
\end{matrix}
\right]
\left[
\begin{matrix}
   T^{k}  &      \widehat{T}  \\
       0   &       0
\end{matrix}
\right]
U^{*}
\\
\nonumber
&
=
GU
\left[
\begin{matrix}
   T^{*}G_{1}T^{k}  &      T^{*}G_{1}\widehat{T}  \\
      S^{*}G_{1}T^{k}+N^{*}G_{3}T^{k}   &       S^{*}G_{1}\widehat{T}+N^{*}G_{3}\widehat{T}
\end{matrix}
\right]
U^{*},
\end{align}
and
\begin{align}
\nonumber
B^{\sim}A^{k}
&=
GU
\left[
\begin{matrix}
   T^{*}  &      0  \\
       (S-G_{1}^{-1}G_{2}P+G_{1}^{-1}G_{2}N)^{*}   &       P^{*}
\end{matrix}
\right]
U^{*}GU
\left[
\begin{matrix}
   T^{k}  &      \widehat{T}  \\
       0   &       0
\end{matrix}
\right]
U^{*}
\\
\nonumber
&
=
GU
\left[
\begin{matrix}
   T^{*}  &      0  \\
       (S-G_{1}^{-1}G_{2}P+G_{1}^{-1}G_{2}N)^{*}   &       P^{*}
\end{matrix}
\right]
\left[
\begin{matrix}
   G_{1}  &      G_{2}  \\
       G_{3}   &       G_{4}
\end{matrix}
\right]
\left[
\begin{matrix}
   T^{k}  &      \widehat{T}  \\
       0   &       0
\end{matrix}
\right]
U^{*}
\\
\nonumber
&
=
GU
\left[
\begin{matrix}
   T^{*}G_{1}T^{k}  &      T^{*}G_{1}\widehat{T}  \\
      S^{*}G_{1}T^{k}+N^{*}G_{2}^{*}T^{k}   &       S^{*}G_{1}\widehat{T}+N^{*}G_{2}^{*}\widehat{T}
\end{matrix}
\right]
U^{*}.
\end{align}
By applying
Remark \ref{WZHBAP-MT1-Remark},
we have $A^{\sim}A^{k}=B^{\sim}A^{k}$.
\end{proof}

\begin{theorem}
\label{WWL-MCE-th5.5}
Let $A, B\in\mathbb{C}_{n,n}$,
 ${\mbox{\rm Ind}}(A)=k$,
  ${\mbox{\rm Ind}}(B)=t$,
${\mbox{\rm rk}}\left(A^{k}\right)
={\mbox{\rm rk}}\left(\left(A^{k}\right)^{\sim}A^{k}\right)=r$
and
${\mbox{\rm rk}}\left(B^{t}\right)
={\mbox{\rm rk}}\left(\left(B^{t}\right)^{\sim}B^{t}\right)=s\geq r$.
 Then
 {{ \begin{align}
\label{WWL-MCE-40}
A \mathop
 \leq \limits ^ {\tiny\textcircled{E}} B\Leftrightarrow
\widehat{A_{1}}
\mathop
\leq \limits ^ {\tiny\textcircled{m}}
\widehat{B_{1}},
\end{align}}}
where
$A=\widehat{A_{1}}+\widehat{A_{2}}$
and
$B=\widehat{B_{1}}+\widehat{B_{2}}$
are
the ${\mathfrak{m}}$-core-EP decompositions of
$A$ and $B$,
respectively.
\end{theorem}

\begin{proof}
Let $A$ be as in (\ref{WWL-MCE-5}).
Applying (\ref{MC-1-3})
and (\ref{WWL-MCE-th3.2}),
then
\begin{align}
\nonumber
\widehat{A_{1}}
&=
AA^{\tiny\textcircled{E}}A
\\
\nonumber
&
=
\widehat{U}
\left[
\begin{matrix}
   T &     S_{1} & S_{2}   \\
       0 &  \widehat{N_{11}} &    \widehat{N_{12}}\\
       0& \widehat{N_{13}} & \widehat{N_{14}}
\end{matrix}
\right]
\left[
\begin{matrix}
   T^{-1}G_{1}^{-1} &     0  &0  \\
       0 &       0  & 0 \\
       0 & 0 & 0
\end{matrix}
\right]
\left[
\begin{matrix}
   G_{1} &     G_{21}  &G_{22}  \\
       G_{31} &       G_{41}  & G_{42} \\
       G_{32} & G_{43} & G_{44}
\end{matrix}
\right]
\left[
\begin{matrix}
   T &     S_{1} & S_{2}   \\
       0 &  \widehat{N_{11}} &     \widehat{N_{12}}\\
       0& \widehat{N_{13}} & \widehat{N_{14}}
\end{matrix}
\right]
\widehat{U}^{*}
\\
\nonumber
&
=
\widehat{U}
\left[
\begin{matrix}
   T &   S_{1}+G_{1}^{-1}(G_{21}\widehat{N_{11}}+G_{22}\widehat{N_{13}})  & S_{2}+G_{1}^{-1}(G_{21}\widehat{N_{12}}+G_{22}\widehat{N_{14}}) \\
       0 &       0 & 0  \\
       0 &      0 & 0
\end{matrix}
\right]
\widehat{U}^\ast
\end{align}
and
\begin{align}
\nonumber
\widehat{A_{2}}
&=
A-\widehat{A_{1}}
\\
\nonumber
&
=
\widehat{U}
\left[
\begin{matrix}
   0 &   -G_{1}^{-1}(G_{21}\widehat{N_{11}}+G_{22}\widehat{N_{13}})  & -G_{1}^{-1}(G_{21}\widehat{N_{12}}+G_{22}\widehat{N_{14}}) \\
       0 &       \widehat{N_{11}} & \widehat{N_{12}}  \\
       0 &      \widehat{N_{13}} & \widehat{N_{14}}
\end{matrix}
\right]
\widehat{U}^\ast.
\end{align}
Write
\begin{align}
\label{WWL-MCE-43}
\alpha
&
=
S_{1}
+G_{1}^{-1}(G_{21}\widehat{N_{11}}
-G_{21}\widetilde{T}
+G_{22}\widehat{N_{13}}),
\\
\label{WWL-MCE-44}
\beta
&
=
S_{2}+G_{1}^{-1}(G_{21}\widehat{N_{12}}
-G_{21}\widehat{S}
+G_{22}\widehat{N_{14}}
-G_{22}\widehat{N}).
\end{align}
{{ Let $B$
be as in (\ref{WWL-MCE-5.7}).
Applying (\ref{WWL-MCE-th3.2}),}}
we have
\begin{align}
\label{WWL-MCE-45}
B^{{\tiny\textcircled{E}}}
&
=
\widehat{U}
\left[
\begin{matrix}
       \left[\begin{matrix}
       T^{-1}   &
       - T^{-1} \alpha \widetilde{T}^{-1}
        \\
       0  &       \widetilde{T}^{-1}
\end{matrix}
\right] \widehat{G_1 }^{-1}
&
0\\
   0  &       0
\end{matrix}
\right]
\widehat{U}^\ast G.
\end{align}
By applying
{{(\ref{WWL-MCE-5.7})}}
and
(\ref{WWL-MCE-45}),
we obtain
\begin{align}
\label{WWL-MCE-46}
\nonumber
\widehat{B_{1}}
&=
BB^{\tiny\textcircled{E}}B
\\
\nonumber
&
=
\widehat{U}
\left[
\begin{matrix}
   T &     \alpha & \beta   \\
       0 & \widetilde{T} &    \widehat{S}\\
       0& 0 & \widehat{N}
\end{matrix}
\right]
\widehat{U}^{*}\widehat{U}
\left[
\begin{matrix}
       \left[\begin{matrix}
       T^{-1}   &
       - T^{-1} \alpha \widetilde{T}^{-1}
        \\
       0  &       \widetilde{T}^{-1}
\end{matrix}
\right] \widehat{G_1 }^{-1}
&
0\\
     0 &       0
\end{matrix}
\right]
\widehat{U}^\ast G\widehat{U}
\left[
\begin{matrix}
   T &     \alpha & \beta   \\
       0 & \widetilde{T} &    \widehat{S}\\
       0& 0 & \widehat{N}
\end{matrix}
\right]
\widehat{U}^{*}
\\
\nonumber
&
=
\widehat{U}
\left[
\begin{matrix}
\left[\begin{matrix}
T  &\alpha\\
0  &   \widetilde{T}
\end{matrix}
\right]
&
\left[\begin{matrix}
      \beta
        \\
       \widehat{S}
\end{matrix}
\right]\\
       0 &       \widehat{N}
\end{matrix}
\right]
\left[
\begin{matrix}
       \left[\begin{matrix}
       T^{-1}   &
       - T^{-1} \alpha \widetilde{T}^{-1}
        \\
       0  &       \widetilde{T}^{-1}
\end{matrix}
\right] \widehat{G_1 }^{-1}
&
0\\
     0 &       0
\end{matrix}
\right]
\left[
\begin{matrix}
   \widehat{G_{1}} &     \widehat{G_{2}}   \\
      \widehat{G_{3}} &     \widehat{G_{4}}
\end{matrix}
\right]
\left[
\begin{matrix}
\left[\begin{matrix}
T   &
\alpha
\\
0  &    \widetilde{T}
\end{matrix}
\right]
&
\left[\begin{matrix}
      \beta
        \\
       \widehat{S}
\end{matrix}
\right]\\
      0 &       \widehat{N}
\end{matrix}
\right]
\widehat{U}^{*}
\\
&
=
\widehat{U}
\left[
\begin{matrix}
\left[\begin{matrix}
T   &
\alpha
\\
0  &    \widetilde{T}
\end{matrix}
\right]
&
\left[\begin{matrix}
      \beta
        \\
       \widehat{S}
\end{matrix}
\right]+\widehat{G_{1}}^{-1}\widehat{G_{2}}\widehat{N}
\\
 0 &       {{ 0}}
\end{matrix}
\right]
\widehat{U}^{*}.
\end{align}
Then
\begin{align}
\widehat{B_{2}}
&=
B-\widehat{B_{1}}
=
\widehat{U}
\left[
\begin{matrix}
   0 &     -\widehat{G_{1}}^{-1}\widehat{G_{2}}\widehat{N}    \\
       0 & \widehat{N}
\end{matrix}
\right]
\widehat{U}^{*}.
\end{align}
Applying
\begin{align}
\widehat{A_{1}}^{\tiny\textcircled{m}}
&=
\widehat{U}
\left[
\begin{matrix}
   T^{-1}G_{1}^{-1} &    0 & 0    \\
       0 & 0 & 0 \\
       0 & 0 & 0
\end{matrix}
\right]
\widehat{U}^{*}G,
\end{align}
we get
\begin{align}
\nonumber
\widehat{A_{1}}^{\tiny\textcircled{m}}
\widehat{B_{2}}
&
=
\widehat{U}
\left[
\begin{matrix}
   T^{-1}G_{1}^{-1} &     0  &0  \\
       0 &       0  & 0 \\
       0 & 0 & 0
\end{matrix}
\right]
\widehat{U}^{*}G\widehat{U}
\left[
\begin{matrix}
   0 &     -\widehat{G_{1}}^{-1}\widehat{G_{2}}\widehat{N}   \\
       0 &  \widehat{N}
\end{matrix}
\right]
\widehat{U}^{*}
\\
\nonumber
&
=
\widehat{U}
\left[
\begin{matrix}
       \left[\begin{matrix}
       T^{-1}G_{1}^{-1}   &
       0
        \\
       0  &       0
\end{matrix}
\right]
&
0\\
0 &       0
\end{matrix}
\right]
\left[
\begin{matrix}
   \widehat{G_{1}} &     \widehat{G_{2}}   \\
      \widehat{G_{3}} &     \widehat{G_{4}}
\end{matrix}
\right]
\left[
\begin{matrix}
   0 &     -\widehat{G_{1}}^{-1}\widehat{G_{2}}\widehat{N}   \\
       0 &  \widehat{N}
\end{matrix}
\right]
\widehat{U}^{*}
\nonumber
=
0,
\\
\nonumber
\widehat{B_{2}}\widehat{A_{1}}^{\tiny\textcircled{m}}
&
=
\widehat{U}
\left[
\begin{matrix}
   0 &     -\widehat{G_{1}}^{-1}\widehat{G_{2}}\widehat{N}   \\
       0 &  \widehat{N}
\end{matrix}
\right]
\widehat{U}^{*}\widehat{U}
\left[
\begin{matrix}
   T^{-1}G_{1}^{-1} &     0  &0  \\
       0 &       0  & 0 \\
       0 & 0 & 0
\end{matrix}
\right]
\widehat{U}^{*}G
\\
\nonumber
&
=
\widehat{U}
\left[
\begin{matrix}
   0 &     -\widehat{G_{1}}^{-1}\widehat{G_{2}}\widehat{N}   \\
       0 &  \widehat{N}
\end{matrix}
\right]
\left[
\begin{matrix}
       \left[\begin{matrix}
       T^{-1}G_{1}^{-1}   &
       0
        \\
       0  &       0
\end{matrix}
\right]
&
0\\
       0  &       0
\end{matrix}
\right]
\widehat{U}^{*}G
\nonumber
=
0,
\\
\label{WWL-MCE-49}
\nonumber
\widehat{A_{1}}^{\tiny\textcircled{m}}\widehat{A_{2}}
&
=
\widehat{U}
{\footnotesize\left[
\begin{matrix}
   T^{-1}G_{1}^{-1} &    0 & 0    \\
       0 & 0 & 0 \\
       0 & 0 & 0
\end{matrix}
\right]}
\widehat{U}^{*}G\widehat{U}
{\footnotesize\left[
\begin{matrix}
   0 &   -G_{1}^{-1}(G_{21}\widehat{N_{11}}+G_{22}\widehat{N_{13}})  & -G_{1}^{-1}(G_{21}\widehat{N_{12}}+G_{22}\widehat{N_{14}}) \\
       0 &       \widehat{N_{11}} & \widehat{N_{12}}  \\
       0 &      \widehat{N_{13}} & \widehat{N_{14}}
\end{matrix}
\right]}
\widehat{U}^\ast
\\
\nonumber
&
=
\widehat{U}{\footnotesize
\left[
\begin{matrix}
   T^{-1}G_{1}^{-1} &    0 & 0    \\
       0 & 0 & 0 \\
       0 & 0 & 0
\end{matrix}
\right]
\left[
\begin{matrix}
   G_{1} &     G_{21}  &G_{22}  \\
       G_{31} &       G_{41}  & G_{42} \\
       G_{32} & G_{43} & G_{44}
\end{matrix}
\right]
\left[
\begin{matrix}
   0 &   -G_{1}^{-1}(G_{21}\widehat{N_{11}}+G_{22}\widehat{N_{13}})  & -G_{1}^{-1}(G_{21}\widehat{N_{12}}+G_{22}\widehat{N_{14}}) \\
       0 &       \widehat{N_{11}} & \widehat{N_{12}}  \\
       0 &     \widehat{N_{13}} & \widehat{N_{14}}
\end{matrix}
\right]}
\widehat{U}^\ast
\\
\nonumber
&
=
0
\end{align}
and
\begin{align}
\nonumber
\widehat{A_{2}}\widehat{A_{1}}^{\tiny\textcircled{m}}
&
=
\widehat{U}
{\footnotesize \left[
\begin{matrix}
   0 &   -G_{1}^{-1}(G_{21}\widehat{N_{11}}+G_{22}\widehat{N_{13}})  & -G_{1}^{-1}(G_{21}\widehat{N_{12}}+G_{22}\widehat{N_{14}}) \\
       0 &       \widehat{N_{11}} & \widehat{N_{12}}  \\
       0 &     \widehat{N_{13}} & \widehat{N_{14}}
\end{matrix}
\right]}
\widehat{U}^\ast\widehat{U}
{\footnotesize\left[
\begin{matrix}
   T^{-1}G_{1}^{-1} &    0 & 0    \\
       0 & 0 & 0 \\
       0 & 0 & 0
\end{matrix}
\right]}
\widehat{U}^{*}G
\\
\nonumber
&
=
\widehat{U}
{\footnotesize\left[
\begin{matrix}
   0 &   -G_{1}^{-1}(G_{21}\widehat{N_{11}}+G_{22}\widehat{N_{13}})  & -G_{1}^{-1}(G_{21}\widehat{N_{12}}+G_{22}\widehat{N_{14}}) \\
       0 &       \widehat{N_{11}} & \widehat{N_{12}}  \\
       0 &     \widehat{N_{13}} & \widehat{N_{14}}
\end{matrix}
\right]
\left[
\begin{matrix}
   T^{-1}G_{1}^{-1} &    0 & 0    \\
       0 & 0 & 0 \\
       0 & 0 & 0
\end{matrix}
\right]}
\widehat{U}^{*}G
\\
\nonumber
&
=
0.
\end{align}

Since
$A \mathop
\leq \limits ^ {\tiny\textcircled{E}} B$
and
$A^{\tiny\textcircled{E}}
=\widehat{A_{1}}^{\tiny\textcircled{m}}$,
 then
 $\widehat{A_{1}}^{\tiny\textcircled{m}}A
 =
 \widehat{A_{1}}^{\tiny\textcircled{m}}B$ and
 $A\widehat{A_{1}}^{\tiny\textcircled{m}}
 =
 B\widehat{A_{1}}^{\tiny\textcircled{m}}$. For
$A=\widehat{A_{1}}+\widehat{A_{2}}$ and
$B=\widehat{B_{1}}+\widehat{B_{2}}$,
we have
\begin{equation}
\nonumber
%\label{WWL-MCE-51}
\left\{
\begin{aligned}
\widehat{A_{1}}^{\tiny\textcircled{m}}
\widehat{A_{1}}+\widehat{A_{1}}^{\tiny\textcircled{m}}\widehat{A_{2}}
=
\widehat{A_{1}}^{\tiny\textcircled{m}}\widehat{B_{1}}
+\widehat{A_{1}}^{\tiny\textcircled{m}}\widehat{B_{2}}
\\
%\label{WWL-MCE-52}
\widehat{A_{1}}\widehat{A_{1}}^{\tiny\textcircled{m}}
+\widehat{A_{2}}\widehat{A_{1}}^{\tiny\textcircled{m}}
=
\widehat{B_{1}}\widehat{A_{1}}^{\tiny\textcircled{m}}
+\widehat{B_{2}}\widehat{A_{1}}^{\tiny\textcircled{m}}
\end{aligned}\right..
\end{equation}
It follows that
$\widehat{A_{1}}^{\tiny\textcircled{m}}\widehat{A_{1}}
=
\widehat{A_{1}}^{\tiny\textcircled{m}}\widehat{B_{1}}$
and
$\widehat{A_{1}}\widehat{A_{1}}^{\tiny\textcircled{m}}
=
\widehat{B_{1}}\widehat{A_{1}}^{\tiny\textcircled{m}}$,
that is,
$\widehat{A_{1}}
\mathop \leq \limits ^ {\tiny\textcircled{m}}
\widehat{B_{1}}$.
\bigskip

'$\Leftarrow$'
{{ Since $\widehat{A_{1}}
\mathop
\leq \limits ^ {\tiny\textcircled{m}}
\widehat{B_{1}}$,
by applying (\ref{WWL-MCE-3022}), we have  $$\widehat{A_{1}}^{\tiny\textcircled{m}}\widehat{A_{1}}
=
\widehat{A_{1}}^{\tiny\textcircled{m}}\widehat{B_{1}}
\mbox{\ and\ }
\widehat{A_{1}}\widehat{A_{1}}^{\tiny\textcircled{m}}
=
\widehat{B_{1}}\widehat{A_{1}}^{\tiny\textcircled{m}}.$$

By applying (\ref{WWL-MCE-22}), (\ref{WWL-MCE-2222})
and Theorem \ref{WWL-MCE-df4.1}, we have $\widehat{A_{2}}\widehat{A_{1}}^{\tiny\textcircled{m}}=0$, $\widehat{A_{1}}^{\tiny\textcircled{m}}\widehat{A_{2}}=\widehat{A_{1}}^{\tiny\textcircled{m}}\widehat{A_{1}}\widehat{A_{1}}^{\tiny\textcircled{m}}\widehat{A_{2}}
=\widehat{A_{1}}^{\tiny\textcircled{m}}(\widehat{A_{1}}\widehat{A_{1}}^{\tiny\textcircled{m}})^{\sim}\widehat{A_{2}}
=\widehat{A_{1}}^{\tiny\textcircled{m}}(\widehat{A_{1}}^{\tiny\textcircled{m}})^{\sim}\widehat{A_{1}}^{\sim}\widehat{A_{2}}=0$, $\widehat{A_{1}}^{\tiny\textcircled{m}}\widehat{B_{2}}=\widehat{A_{1}}^{\tiny\textcircled{m}}\widehat{B_{1}}\widehat{B_{1}}^{\tiny\textcircled{m}}\widehat{B_{2}}
=\widehat{A_{1}}^{\tiny\textcircled{m}}(\widehat{B_{1}}\widehat{B_{1}}^{\tiny\textcircled{m}})^{\sim}\widehat{B_{2}}
=\widehat{A_{1}}^{\tiny\textcircled{m}}(\widehat{B_{1}}^{\tiny\textcircled{m}})^{\sim}\widehat{B_{1}}^{\sim}\widehat{B_{2}}=0$
and
$\widehat{B_{2}}\widehat{A_{1}}^{\tiny\textcircled{m}}=\widehat{B_{2}}\widehat{B_{1}}^{\tiny\textcircled{m}}\widehat{B_{1}}\widehat{A_{1}}^{\tiny\textcircled{m}}
=0$.
Furthermore,
by applying (\ref{WWL-MCE-th4.2.1-2}),
we have
$A^{\tiny\textcircled{E}}\widehat{A_{1}}
=
A^{\tiny\textcircled{E}}\widehat{B_{1}}$,
$\widehat{A_{1}}A^{\tiny\textcircled{E}}
=
\widehat{B_{1}}A^{\tiny\textcircled{E}}$,
$A^{\tiny\textcircled{E}}\widehat{A_{2}}=\widehat{A_{1}}^{\tiny\textcircled{m}}\widehat{A_{2}}=0$,
$A^{\tiny\textcircled{E}}\widehat{B_{2}}=\widehat{A_{1}}^{\tiny\textcircled{m}}\widehat{B_{2}}=0$,
$\widehat{A_{2}}A^{\tiny\textcircled{E}}=\widehat{A_{2}}\widehat{A_{1}}^{\tiny\textcircled{m}}=0$
and
$\widehat{B_{2}}A^{\tiny\textcircled{E}}=\widehat{B_{2}}\widehat{A_{1}}^{\tiny\textcircled{m}}=0$.
Therefore,
it follows that
$A^{\tiny\textcircled{E}}A=A^{\tiny\textcircled{E}}\widehat{A_{1}}+A^{\tiny\textcircled{E}}\widehat{A_{2}}
=A^{\tiny\textcircled{E}}\widehat{A_{1}}
=A^{\tiny\textcircled{E}}\widehat{B_{1}}=A^{\tiny\textcircled{E}}\widehat{B_{1}}+A^{\tiny\textcircled{E}}\widehat{B_{2}}
=A^{\tiny\textcircled{E}}B$
and
$AA^{\tiny\textcircled{E}}=\widehat{A_{1}}A^{\tiny\textcircled{E}}+\widehat{A_{2}}A^{\tiny\textcircled{E}}
=\widehat{A_{1}}A^{\tiny\textcircled{E}}
=\widehat{B_{1}}A^{\tiny\textcircled{E}}=\widehat{B_{1}}A^{\tiny\textcircled{E}}+\widehat{B_{2}}A^{\tiny\textcircled{E}}
=BA^{\tiny\textcircled{E}}$, that is, $A
\mathop \leq \limits ^ {\tiny\textcircled{E}}
B$.}}
\end{proof}

\begin{example}
Let
\begin{align}
\label{WWL-MCE-53}
A=
\left[
\begin{matrix}
 1 &  2 & 3\\
   0     & 0  & 1\\
   0  & 0 & 0
\end{matrix}
\right]\
and\
B=
\left[
\begin{matrix}
 1     &  2  &3\\
   0   &  0  &0\\
   0   &0    & 0
\end{matrix}
\right].
\end{align}
Then $A \mathop \leq \limits ^ {\tiny\textcircled{E}} B$
and $B \mathop \leq \limits ^ {\tiny\textcircled{E}} A$.
However, $A\neq B$.
Therefore,
the {${\mathfrak{m}}$}-core-EP order is not antisymmetric.
\end{example}

\begin{theorem}
The {${\mathfrak{m}}$}-core-EP order
is not a partial order
but merely a pre-order.
\end{theorem}
\begin{proof}
Reflexivity of the relation is obvious.
Suppose
$A \mathop \leq \limits ^ {\tiny\textcircled{E}} B$
and
$B \mathop \leq \limits ^ {\tiny\textcircled{E}} C$,
in which
$A=\widehat{A_{1}}+\widehat{A_{2}}$,
$B=\widehat{B_{1}}+\widehat{B_{2}}$
and
$C=\widehat{C_{1}}+\widehat{C_{2}}$
 are
 the ${\mathfrak{m}}$-core-EP decomposition
 of $A$, $B$ and $C$,
respectively.
Then $\widehat{A_{1}}
\mathop \leq \limits ^ {\tiny\textcircled{m}}
\widehat{B_{1}}$
and
$\widehat{B_{1}}
\mathop \leq \limits ^ {\tiny\textcircled{m}}
\widehat{C_{1}}$.
Therefore,
$\widehat{A_{1}} \mathop \leq \limits ^ {\tiny\textcircled{m}}
\widehat{C_{1}}$.
By applying
Theorem \ref{WWL-MCE-th5.5},
we have
$A
\mathop \leq \limits ^ {\tiny\textcircled{E}}
C$.
\end{proof}

\section*{ Acknowledgements}
The authors wish to extend  their sincere gratitude to the referees for their precious comments and suggestions.

\section*{Disclosure statement}
No potential conflict of interest was reported by the authors.

\section*{Funding}
The first  author
 was supported partially by the Research Fund Project of Guangxi University for Nationalities
  (No. 2019KJQD03),
 Guangxi Natural Science Foundation
 (No. 2018GXNSFAA138181)
 and the Special Fund for Bagui Scholars of Guangxi
 (No. 2016A17).
 The  second
 author was supported partially by
the Special Fund for Science and Technological Bases and Talents of Guangxi
 (No. GUIKE AD19245148) and
 New Centaury National Hundred, Thousand and Ten Thousand Talent Project of Guangxi
 (No. GUIZHENGFA210647HAO).
The  third author
was supported partially by the National Natural Science Foundation of China
 (No. 12061015), Guangxi Natural Science Foundation (No. 2018GXNSFDA281023),
 and High Level Innovation Teams and Distinguished Scholars in Guangxi Universities (No. GUIJIAOREN201642HAO).

%

%\section*{References}

\end{document}